\documentclass[10pt,a4paper,reqno]{amsart}

\usepackage[utf8]{inputenc}

\usepackage{array,amsbsy,amscd,amsfonts,color,amssymb,amstext,amsmath,latexsym,amsthm,amscd,multicol,graphicx, tikz, physics}
\usepackage[english]{babel}
\usepackage{hyperref}
\usepackage[mathscr]{eucal}
\usepackage{tikz-cd}

\newtheorem{theorem}{Theorem}
\numberwithin{theorem}{section}
\newtheorem{proposition}[theorem]{Proposition}
\newtheorem{lemma}[theorem]{Lemma}
\newtheorem{corollary}[theorem]{Corollary}
\newtheorem{question}[theorem]{Question}

\theoremstyle{definition}\newtheorem{definition}[theorem]{Definition}
\theoremstyle{definition}\newtheorem{remark}[theorem]{Remark}
\theoremstyle{definition}\newtheorem{example}[theorem]{Example}
\theoremstyle{definition}\newtheorem*{notation*}{Notation}
\theoremstyle{definition}\newtheorem*{convention*}{Convention}

\numberwithin{equation}{section}

\newcommand{\N}{\mathbb{N}}

\newcommand{\A}{\mathcal{A}}
\newcommand{\h}{\mathcal{H}}
\newcommand{\K}{\mathcal{K}}

\newcommand{\p}{\mathcal{P}}
\newcommand{\T}{\mathbb{T}}

\newcommand{\hpi}{\mathcal{H}_{\pi}}
\newcommand{\B}{\mathcal{B}}
\newcommand{\s}{\mathcal{P}}
\newcommand{\bh}{\mathcal{B}(\mathcal{H})}
\newcommand{\bk}{\mathcal{B}(\mathcal{K})}

\newcommand{\sqd}{D^{1/2}}
\newcommand{\cst}{C^\ast}

\newcommand{\C}{\mathbb{C}}

\newcommand{\ox}{\mathcal{O}(X)}
\newcommand{\cx}{C(X)}
\newcommand{\cbx}{C_b(X)}
\newcommand{\px}{\mathcal{P}_{\h}(X)}
\newcommand{\uhx}{UCP_{\mathcal{H}}(C(X))}
\newcommand{\rpx}{\mathcal{R}\mathcal{P}_{\h}(X)}
\newcommand{\cpx}{\mathcal{C}\mathcal{P}_{\h}(X)}
\newcommand{\piox}{\pi\left(\mathcal{O}(X)\right)}

\DeclareMathOperator{\Span}{span}
\DeclareMathOperator{\cspan}{\overline{span}}
\DeclareMathOperator{\ran}{ran}
\DeclareMathOperator{\cran}{\overline{ran}}
\DeclareMathOperator{\Ext}{Ext}
\DeclareMathOperator{\co}{co}

\makeatletter
\@namedef{subjclassname@2020}{2020 Mathematics Subject Classification}
\makeatother

\begin{document}

\title[POVMs and UCP maps]{ $C^*$-extreme points of positive operator valued measures and unital completely positive maps}

\author{Tathagata Banerjee}
\email{tathagatabanerjee85@gmail.com }

\author{B.V. Rajarama Bhat}
\email{bhat@isibang.ac.in}

\author{Manish Kumar}
\email{ manish\_rs@isibang.ac.in }

\address{Indian Statistical Institute, Stat-Math. Unit, R V College Post, Bengaluru 560059, India}

\begin{abstract}
We study the quantum  ($C^*$)  convexity structure of  normalized
positive operator valued measures (POVMs) on measurable spaces. In
particular, it is seen that unlike extreme points under classical
convexity, $C^*$-extreme points of normalized POVMs on countable
spaces (in particular for finite sets) are always spectral measures
(normalized projection valued measures). More generally it is shown
that atomic $C^*$-extreme points are spectral. A Krein-Milman type
theorem for POVMs has also been proved. As an application it is
shown that a map on any commutative unital $C^*$-algebra with
countable spectrum (in particular $\C^n$) is $C^*$-extreme in the
set of unital completely positive maps if and only if it is a unital
$*$-homomorphism.
\end{abstract}

\maketitle

\section{Introduction}

The classical notion of convexity  plays an important role in
analysis in understanding various mathematical structures. Often the
problem is to identify extreme points of a convex set. Once that is
done, subsequently one may try to show that all points of the set
are convex combinations of extreme points or their limits. There
have been several approaches to generalize the notion of convexity
to have a non-commutative (or quantum) variant, for example
CP-convexity \cite{Fujimoto}, matrix convexity \cite{Effros
Winkler}, nc-convexity \cite{Davidson Kennedy} and $C^*$-convexity
(\cite{Loebl Paulsen}, \cite{Farenick Morenz}).

One prominent and  useful idea is to replace  positive scalars in
the interval $[0,1]$ as coefficients for convexity by positive,
contractive and invertible elements in a $C^*$-algebra. This is the
notion of quantum or $C^*$-convexity. The study of $C^*$-convexity
and $C^*$-extreme points seems to have been started by Loebl and
Paulsen \cite{Loebl Paulsen} for  subsets of $C^*$-algebras and
subsequently  many researchers have explored it. Farenick and Morenz
\cite{Farenick Morenz} defined and initiated a study of
$C^*$-convexity and $C^*$-extreme points (see Definition
\ref{C*-convexity for UCP maps}) for unital completely positive
(UCP) maps on $C^*$-algebras taking values in the algebra $\bh $ of
all bounded operators on a Hilbert space $\h .$  They call these
maps as {\em generalized states} following an earlier convention, as
UCP maps taking values in $\bh $ with $\h $ one dimensional  are
just states. In particular they show that for $n\in {\mathbb N}$,
$C^*$-extreme UCP maps on the $C^*$-algebra $\mathbb {C}^n$, taking
values in matrices (that is, $\bh $ with finite dimensional $\h$)
are $*$-homomorphisms. Whether the same conclusion can be arrived at
when the space $\h $ is infinite dimensional and separable was left
open. We settle it here affirmatively in Theorem \ref{$C^*$-extreme
UCP maps are homomorphisms}.

It should be  mentioned here that there are several papers analyzing
$C^*$-convexity of UCP maps:  \cite{Farenick Morenz}, \cite{Farenick
Zhou},  \cite{Gregg}, \cite{Zhou} and \cite{Magajna}, to name a few.
In \cite{Farenick Zhou} and \cite{Zhou}, one can see some abstract
characterizations of $C^*$-extreme points of UCP maps. There is a
well-known relationship (see \cite{Hadwin}, \cite{Paulsen})  between
UCP maps on the $C^*$-algebra $C(X)$ of continuous functions on a
compact Hausdorff space $X$ and positive operator valued measures
(POVMs) on the Borel $\sigma $-algebra $\ox $ on $X.$ Many authors
while studying UCP maps on commutative $\cst$-algebras exploit this
relationship. We follow the same approach and for the purpose first
study POVMs.

Positive operator valued measures (POVMs) are called generalized
measurements in quantum mechanics and are basic mathematical tools
in quantum information theory. There is extensive literature on
POVMs and we do not attempt a survey. Some standard references are
\cite{Davies}, \cite{Holevo}, \cite{Schroeck},   \cite{Davies
Lewis} and \cite{Han Larson Liu}. The notions of $C^*$-convexity
and $C^*$-extreme points have natural extensions to POVMs (see
Definition \ref{definition of C*-convex sets} and \ref{definition of
C*-extreme points}). Here we study $C^*$-convexity of POVMs on a
measurable space $(X,\ox)$, where $\ox$ is a $\sigma$-algebra of
subsets of a set $X$.  The problem of identifying $C^*$-extreme
points of POVMs has been open for several decades even for finite
sets. The result from 1997 of Farenick and Morenz \cite{Farenick
Morenz}  translates to saying that $C^*$-extreme positive {\em
matrix} valued measures on a finite set $X$ are spectral measures (normalized
projection valued measures).  We
generalize the result of \cite{Farenick Morenz} considerably, as we
allow general POVMs on all countable
spaces and still all the $C^*$-extreme points are spectral (Theorem
\ref{atomic $C^*$-extreme points are PVM}). This is important
because it is in stark contrast with classical (linear) convexity.
Extreme points of POVMs under classical convexity are not
necessarily spectral measures and are hard to describe even
for finite sets, though abstract characterizations are available.
$C^*$-extreme points being spectral measures have physical
significance as they relate to classical measurements. Our result
reinforces the idea that $C^*$-convexity is perhaps the suitable
notion of convexity in the quantum setting.

Our main goal is to explore the $C^*$-convexity  structure and
identify the $C^*$-extreme points of  POVMs taking values for
arbitrary separable Hilbert spaces. We shall also present some
results on usual (classical) extreme points of POVMs for comparison.
We investigate POVMs  via decomposing them into a sum of atomic and
non-atomic POVMs.  Some of these results on POVMs could be folklore
in the literature, but we present them here for clarity of
presentation and for completeness.

This paper is organized as follows. We start with the definition of
POVMs on measurable spaces in Section \ref{Basic Properties of
POVMs} and state some known basic results such as  Naimark's dilation
theorem, Radon-Nikodym type theorem and so on. A brief description
of atomic and non-atomic POVMs is given. In Section \ref{main
results on C*-extreme points}, we present some of our main results
on $C^*$-extreme points. The most crucial technical step is in the
proof of Theorem \ref{if mu(A) commutes with everything, then it is
a projection}.  Heinosaari and  Pellonp\"{a}\"{a} \cite{Heinosaari
Pellonpaa} have shown that extreme points of POVMs with commutative
ranges are spectral. The same conclusion holds under $C^*$-convexity
(Theorem \ref{commutative C*-extreme points are PVM}) as well.  Most importantly  all atomic $\cst$-extreme points are
also seen to be spectral (Theorem \ref{atomic $C^*$-extreme points
are PVM}). This
also helps us in proving that $\cst$-extreme points are spectral for
finite dimensional Hilbert spaces, which we prove in full
generality.

In Section \ref{mutually singular povm}, a notion of disjointness
for spectral measures is introduced and we see that it is equivalent to  mutual
singularity. We study the behaviour of $\cst$-extreme points on
taking direct sums of mutually singular POVMs. In particular, we
show that every $\cst$-extreme point decomposes into a direct sum of
an atomic POVM and a non-atomic POVM, mutually singular to each
other.  Next in Section \ref{measure isomorphism}, we explore basic
properties like $\cst$-convexity, atomicity etc under a notion of
measure isomorphism of POVMs. In Section \ref{povm on topological
space}, we analyze  POVMs on topological spaces. In this case, we
consider the notion of regularity of POVMs and obtain some results
analogous to  classical measure theory. We also consider a topology
on the collection of all POVMs and prove a Krein-Milman type theorem
(Theorem \ref{Krein-Milman type theorem for POVM}). Lastly in
Section \ref{application to completely positive maps}, we describe a
well-known correspondence between regular POVMs on a compact
Hausdorff space $X$ and completely positive maps on  the space
$C(X)$ of all continuous functions on $X$. Using the results got
earlier for POVMs  and this correspondence, we obtain a number of
results for  UCP maps on $\cx$. In particular we show that
$C^*$-extreme maps on commutative unital  $C^*$-algebras with countable spectrum
are $*$-homomorphisms (Theorem \ref {$C^*$-extreme UCP maps are
homomorphisms}). Then making use of the theory of measure
isomorphism of POVMs, we show that separable commutative unital
$C^*$-algebras with uncountable spectrum always admit non
$*$-homomorphic UCP maps as $C^*$-extreme points (Theorem \ref{a non
homomorphic C*-extreme point for separable C*-algebra}). We also
show a Krein-Milman type theorem for the collection of all UCP maps
on $\cx$ equipped with bounded-weak topology. In the concluding
section we remark that the study of $C^*$-convexity can easily be
extended to the setting of locally compact Hausdorff spaces by
taking one point compactifications. We  end with a question on
identifying $C^*$-extreme points of unital completely positive maps
on the $C^*$-algebra $l^{\infty }.$

It may be remarked here that, although we have relegated a detailed
description of the relationship between POVMs and completely
positive maps to Section \ref{application to completely positive
maps}, occasionally even in earlier sections we would be making
references to some known results from the theory of completely
positive maps.

\begin{convention*}
 All Hilbert spaces on which POVMs and UCP maps act will be complex and separable as that is where our interest lies. However,
when we consider Naimark's dilation of POVMs or Stinespring representations of UCP maps
 we may end up with non-separable Hilbert spaces and this has to be kept in mind. We follow the convention of the inner product being linear in the second variable. Throughout   $\bh$
denotes  the algebra of all bounded operators on a complex separable
Hilbert space $\h.$ If $\h , {\mathcal K}$ are two Hilbert spaces,
$\mathcal{B}(\h ,{\mathcal K})$ denotes the space of all  bounded linear
operators from $\h $ to  ${\mathcal K}.$ For a subset $M$ of a
Hilbert space, $[M]$ denotes the closed subspace generated by $M$.
For any map $f$, $\ran(f)$ denotes its range. Usually $A,B,C$
etc. will denote measurable subsets of general measurable spaces. Terms like
$\mu,\nu$ etc will denote arbitrary POVMs while $\pi,\rho$ will be
used specifically for spectral measures. The Hilbert space on which
a spectral measure $\pi$ acts will usually be denoted (and taken
without mention) by $\hpi$. Terms like
$\phi,\psi$ etc will denote completely positive maps on a
$\cst$-algebra. By a positive measure, we mean a (not necessarily
finite) usual scalar valued measure taking value in $[0,\infty]$. For our convenience, we always assume that singleton sets are measurable.
\end{convention*}

\section{Basic Properties of POVMs}\label{Basic Properties of POVMs}

\subsection{Positive operator valued measures}
In this section, we recall the definition and some basics of
positive operator valued measures. This would also help us in fixing
the notation. See \cite{Davies}, \cite{Holevo}, \cite{Paulsen},
\cite{Schroeck} and \cite{Han Larson Liu} for general references.

Unless stated otherwise, $X$ is a non-empty set and $\mathcal{O}(X)$ denotes a $\sigma$-algebra of subsets of $X$.
 The pair $(X,\ox)$ is called a {\em measurable   space} and  the elements of $\ox$ are called {\em measurable   subsets}. We shall simply call $X$ a measurable space without mentioning the underlying $\sigma$-algebra $\ox$.
 To avoid some unnecessary complications in presentation, we assume
 that all singleton subsets of $X$ are measurable. When $X$ is a topological space, we shall assume $\ox$ to be the Borel $\sigma$-algebra on $X$.
  All topological spaces under consideration would be Hausdorff.

\begin{definition}
Let $X$ be a measurable space and let $\h$ be a Hilbert space.
 A \emph{positive operator valued measure(POVM)} on $X$
with values in  $\bh $ is a map
$\mu:\ox\to\bh$ satisfying the following:
\begin{itemize}
    \item  $\mu(A)\geq 0$ in $\bh$ for all $A\in\ox$ and
    \item  for every $h,k\in\h$, the map $\mu_{h,k}:\ox\to\C$ defined  by
\begin{equation}
    \mu_{h,k}(A)=\langle h,\mu(A)k\rangle~~\text{for all } A\in\ox,\label{eq:notation for mu_h,k}
\end{equation}
is a  complex measure.
\end{itemize}

Moreover, a POVM $\mu$ is called
\begin{enumerate}
    \item  \emph{normalized}  if $\mu(X)=I_\h$, the identity operator on $\h$.
    \item  \emph{projection valued measure (PVM)}
    if $\mu( A)$ is a projection for each $ A\in\ox$.
    \item {\em spectral measure} if $\mu$ is a PVM and is normalized.
\end{enumerate}
\end{definition}

It follows from the definition of POVM that, for any increasing (or
decreasing)  sequence $\{A_n\}$ of measurable  subsets converging to
$A$ i.e. $A_n\subseteq A_{n+1}$ and $\cup_nA_n=A$ (or $A_n\supseteq
A_{n+1}$ and  $\cap A_n=A$), $\mu(A_n)\to \mu(A)$ in weak operator
topology (WOT) in $\bh$. Since convergence of an increasing (or
decreasing) sequence of bounded operators is equivalent for both weak
operator topology and strong operator topology (SOT), it follows
that $\mu(A_n)\to\mu(A)$ in SOT. Also, since on bounded subsets of $\bh$, WOT
and $\sigma$-weak topology  agree, we infer that
$\mu(A_n)\to\mu(A)$ in $\sigma$-weak topology. Therefore, in the
countable additivity of POVM:
$$\mu \left(\bigcup_{n=1}^{\infty} B_n\right) =\sum_{n=1}^\infty\mu (B_n),~~B_n\in
\ox, B_n\cap B_m=\emptyset ~\mbox{for} ~ n\neq m,$$ the convergence
of the series holds in WOT, SOT and $\sigma $-weak topologies. So
for POVMs such sums can be considered in any of the three
topologies.

For any POVM $\mu $, by $\mu _{h,k}$ we would mean the complex measure
defined in \eqref{eq:notation for mu_h,k}. It is clear that a POVM
$\mu $ is determined by its associated family of complex measures $\{ \mu
_{h,k}: h,k\in \h \}.$

\begin{notation*}
 Let $POVM_\h(X)$ denote the
collection of all POVMs  on $\ox$ with values in $\bh $ and let
$\px$ denote the collection of all normalized elements in
$POVM_\h(X).$
\end{notation*}

We frequently make use of the following remarks in subsequent results without always explicitly
referring to them.
\begin{remark}
It is well-known that for a POVM $\mu$, that $\mu(A)$ is a
projection for all $A\in \ox$ (i.e. $\mu$ is a PVM) is equivalent to the fact that
$\mu(B\cap C)=\mu(B)\mu(C)$ for all $B,C\in\ox$ (see pg 34,
\cite{Schroeck}).
\end{remark}

\begin{remark}\label{a collection of disjoint sets with non zero measure must be countable}
Let $\mu:\ox\to\bh$ be a POVM and let $\{B_i\}_{i\in I}$ be a collection of mutually disjoint measurable subsets such that $\mu(B_i)\neq 0$ for each $i\in I$. Then by using separability of $\h$, one can  show that  $I$ is countable (Lemma 3.1, \cite{Dorofeev Graaf}) as follows: consider any strictly positive density operator $S$ on $\h$ such that the map $T\mapsto\tr(ST)$ ($\tr$ denotes trace) is a normal faithful state on $\bh$. Define the positive measure $\mu_S:\ox\to[0,\infty)$ by $\mu_S(A)=\tr(S\mu(A))$ for all $A\in\ox$. Note that $\mu_S(B_i)\neq0$ for all $i\in I$ and since $\sum_{i\in I}\mu_S(B_i)\leq\mu_S(\cup_{i\in I}B_i)<\infty$, we conclude that $I$ is countable.
\end{remark}

\subsection{Naimark's  dilation theorem}
The classical dilation theorem of Naimark \cite{Neumark} shows that
POVMs can be dilated to spectral measures:   Let $X$ be a measurable space
and $\mu:\ox\to\bh$ be a POVM. Then there exists a triple
$(\pi,V,\hpi)$ where $\hpi$ is a Hilbert space,
$\pi:\ox\to\mathcal{B}(\hpi)$ is a spectral measure
and $V\in\mathcal{B}(\h,\hpi)$ such that
\begin{equation} \label{Naimark Dilation theorem}
    \mu( A)=V^*\pi( A)V  ~~~\text{for all } A\in \ox
\end{equation}
and the minimality condition: $\hpi=[\pi(\ox)V\h]$ is satisfied.
Moreover such a dilation is unique up to unitary equivalence. The
triple $(\pi, V, \hpi )$ is called a \emph{Naimark dilation triple}
for $\mu$. Since $\pi$ is spectral, note  from (\ref{Naimark
Dilation theorem}) that $V$ is an isometry if and only if $\mu $ is a
normalized POVM.

Naimark's theorem  is text book material. The proof generally
uses the usual GNS construction method. Some possible references are
(Theorem II.11.F,  \cite{Schroeck}) and (Theorem 2.1.1,
\cite{Holevo}).  A proof using  Stinespring's theorem for
completely positive maps is also well-known (Theorem 4.6,
\cite{Paulsen}), but then POVMs under consideration will have  to be
assumed to be regular on the Borel $\sigma-$algebra of some  locally compact Hausdorff space. As an immediate application of
Naimark's dilation theorem we have the following result. Here and
elsewhere, $\mathcal{M}'$ denotes the commutant of a subset
$\mathcal{M}$ in $\bh$.

\begin{proposition}\label{prop:projection commutes with everything}
Let $\mu:\ox\to\mathcal{B}(\mathcal{H})$ be a normalized POVM and
$\mu(E)$  a projection for some $E\in\ox$. Then $\mu(E\cap
A)=\mu(E)\mu(A)=\mu(A)\mu(E)$ for every $A\in\ox$. In particular,
$\mu(E)\in\mu(\ox)'$ and hence $\ran(\mu(E))$, the range of $\mu(E)$ is a reducing
subspace for all $\mu(A)$, $A\in\ox$.
\end{proposition}
\begin{proof}
Let $(\pi,V,\hpi)$ be the minimal Naimark dilation for $\mu$. As
noticed earlier, since $\mu $ is normalized and $\pi$ is spectral, it follows that $V$ is an isometry.
Now for any $A\in\ox$, as $\mu (A)= V^*\pi(A)V$ and $V^*V=I_\h$, we get
\begin{align*}
    [V\mu(A)-\pi(A)V]^*\cdot[V\mu(A)-\pi(A)V]
    &= [\mu(A)V^*-V^*\pi(A)]\cdot[V\mu(A)-\pi(A)V] \\
    & =\mu(A)^2-\mu(A)^2-\mu(A)^2+\mu(A)\\
    &    = \mu(A)^2-\mu(A).
\end{align*}
In particular, since $\mu(E)$ is a projection, we get
$V\mu(E)=\pi(E)V$. For any
$A\in\ox$, therefore
\[\mu(A)\mu(E)=V^*\pi(A)V\mu(E)=V^*\pi(A)\pi(E)V=V^*\pi(A\cap E)V=\mu(A\cap E).\]
Similarly or by taking adjoint of the last equation we get
$\mu(E)\mu(A)=\mu(E\cap A)$.
\end{proof}

\begin{definition}
A POVM $\mu$ is \emph{concentrated} on a measurable subset $E$ if
$\mu(A)=\mu(A\cap E)$ for all $A\in\ox$.
\end{definition}

 Note that a POVM $\mu$ being concentrated on a subset $E$ just means that $\mu (X\setminus E)=0$. This is not same
as saying that $E$ is the support of $\mu$. In fact when $X$ is a
topological space, the {\em support} of $\mu$ is defined as the smallest closed
subset $C$ such that $\mu(C)=\mu(X)$.

\begin{proposition}\label{zero sets of mu and pi are same}
Let $\mu:\ox\to\bh$ be a POVM with the minimal Naimark dilation
$(\pi,V,\hpi)$. Then for any $A\in \ox$, $\mu(A)=0$ if and only if
$\pi(A)=0$. In particular, $\mu $ is concentrated on $E\in \ox$ if
and only if $\pi $ is concentrated on $E.$
\end{proposition}
\begin{proof}
Let $\mu(A)=0$. Then for any $B\in\ox$ and $h\in \h$, we get
 \begin{align*}
     \langle \pi( A)\pi(B)Vh, \pi(B)Vh\rangle=\langle V^* \pi(B\cap A)Vh,h\rangle=\langle\mu(B\cap A)h,h\rangle\leq\langle \mu( A)h,h\rangle=0.
 \end{align*}
 Since $\{\pi(B)Vh;h\in\h,B\in\ox\}$ is total in $\hpi$ by the minimality condition,  we
  conclude that $\pi( A)=0$. The converse is obvious. The second assertion follows from the first.
\end{proof}

\begin{remark}
As we have already mentioned in  Convention, all Hilbert spaces on which  POVMs  act are assumed to be separable. But note that the Hilbert space $\hpi$ in the minimal Naimark dilation $(\pi,V,\hpi)$ of a POVM need not always be separable. Nevertheless, notions like atoms and atomic/non-atomic POVMs (Definition \ref{definition of atomic and non-atomic POVM}), mutual singularity (Definition \ref{definition of mutual singularity}) of POVMs, regularity (Definition \ref{definition of regularity}) of a POVM etc. do not need the assumption of separability of the Hilbert space and hence will naturally be considered for the spectral measure $\pi$.
\end{remark}

\subsection{Radon-Nikodym type theorem}
In classical measure theory, the Radon-Nikodym derivative of a ($\sigma$-finite) positive
measure absolutely continuous with respect to another ($\sigma$-finite) positive measure is a well-established fact.
 There have been several attempts to generalize it to the case of absolutely continuous POVMs (which is defined in
 a similar way as usual positive measures), especially for finite dimensional Hilbert spaces, see for
 example \cite{Farenick Plosker Smith}, \cite{Mclaren Plosker Ram}. In this paper however we consider a
 different
  notion of comparison of POVMs.  We say $\nu $ is {\em dominated\/} by $\mu $ (denoted by $\nu\leq \mu$)
   if $\mu-\nu$ is a POVM.  Here also a   Radon-Nikodym type of theorem is known and is well
   studied. It is analogous to a Radon-Nikodym theorem for completely positive maps by  Arveson (Theorem 1.4.2, \cite{Arveson1}).
See \cite{Raginsky} for a more recent account of this result of
Arveson and its implications to quantum information theory.

For readers convenience we present an outline of the proof. Here the
operator $D$ can be
 thought of as the Radon-Nikodym derivative of $\nu $ with respect
 to $\mu .$

\begin{theorem}\label{thm:Radon-Nikodym type theorem}(Radon-Nikodym type theorem)
 Let $\mu:\ox\to\bh$ be a POVM with the minimal Naimark dilation $(\pi,V,\hpi)$. Then for a
 POVM $\nu:\ox\to\bh$, $\nu\leq\mu$ (i.e. $\mu-\nu$ is a POVM) if  and only if there exists a positive
 contraction $D\in\pi(\ox)'$ such that $\nu(A)=V^*D\pi(A)V$ for all $A\in\ox.$
\end{theorem}
\begin{proof}
The proof of `if' part is obvious. For the converse, assume that $\mu-\nu$ is a POVM. Let $(\h_{\rho}, \rho, W)$ be the minimal Naimark dilation for $\nu$ and define an operator $T:\h_{\pi}\to\h_{\rho}$ as follows: first define $T$ on the subspace $\Span\{\pi(A)Vh;A\in\ox,h\in\h\}$ of $\hpi$ by
$  T\left(\pi(A)Vh\right)=\rho(A)Wh,$ for all $A\in\ox, h\in\h$
and extend it linearly. One can easily show that $T$ is a well-defined contraction by using the fact that $\sum_{i,j=1}^n\langle h_i,(\mu-\nu)(A_i\cap A_j)h_j\rangle\geq0$ for any $A_i\in\ox, h_i\in\h, 1\leq i\leq n$.
So it extends as a
contraction to its closure $\hpi$, which we still denote by $T$.
Set $D=T^*T$. Then $D$ is a positive contraction and it is immediate to verify that $D\in \pi(\ox)'$ and $\nu(A)=V^*D\pi(A)V$ for all $A\in\ox$.
\end{proof}

\subsection{Extreme POVMs}
The set $\px$, which is the collection of all  normalized POVMs on
$X$ with values in $\bh $ is  clearly a convex set. Extreme points
of this set are well studied, especially when $X$ is a finite set or
a compact Hausdorff space and $\h$ is a finite dimensional Hilbert
space (see \cite{KRP1}, \cite{Chiribella Ariano Schlingemann},
\cite{Farenick Plosker Smith} and \cite{Heinosaari Pellonpaa}). In this
paper, we are not focusing much on extreme points of $\px$.
Nevertheless, we provide some results for the sake of comparison
with $\cst$-extreme points. It is to be noted that even when $X$ is
finite with more than two points, the set of extreme points is
difficult to describe. This is true even when $\h$ is finite
dimensional. The following abstract characterization of extreme
points of $\px$ is again inspired by  Arveson's result (Theorem
1.4.6, \cite{Arveson1}) which characterizes the extreme points of
unital completely positive maps on a $C^*$-algebra.  This must have
been noted by several researchers for the case of POVMs  and so we
just outline the proof.

\begin{theorem}\label{thm:extrme point criterion for POVM}(Extreme point criterion)
Suppose that $\mu\in\px$  has  the minimal Naimark
dilation  $(\pi,V,\hpi)$.  Then a necessary and sufficient criterion
for  $\mu$ to be extreme in $\px$ is that the map $D\mapsto V^*DV$
from  $\pi(\ox)'$ to $\bh $ is injective.
\end{theorem}
\begin{proof}
First assume that $\mu$ is extreme in $\px$. Let $V^*DV=0$ for some
$D\in\piox'$. Without loss of generality, we can assume that
$-I_{\hpi}\leq D\leq I_{\hpi}$.  Write $\mu=(\mu^++\mu^-)/2$ where
$\mu^{\pm}(\cdot)=V^*(I_{\hpi}\pm D)\pi(\cdot) V$. Then as $\mu$ is extreme
in $\px$, we must have $\mu=\mu^{+}$. Hence $V^*D\pi(\cdot) V=0$,
which implies  $D=0$. For the converse, assume  the injectivity
of the map $D\mapsto V^*DV$,  and let $\mu=(\mu_1+\mu_2)/2$ for
$\mu_1,\mu_2\in\px$. By Radon-Nikodym type theorem, there are
positive contractions $D_i\in \piox'$, $i=1,2$ such that
$\frac{\mu_i (\cdot )}{2}=V^*D_i\pi(\cdot) V$. But then as $\mu _i$
is normalized, we have
 $V^*(2D_i-I_{\hpi})V=0$ and hence the hypothesis
implies $2D_i=I_{\hpi}$. Thus we get $\mu_i(\cdot)=V^*\pi(\cdot) V=\mu(\cdot)$
for $i=1,2,$ which proves that $\mu$ is extreme in $\px$.
\end{proof}

The following is an immediate corollary of this theorem. It can also
be seen directly, as projections are extremal in the set of positive
contractions.
\begin{corollary}
Every spectral measure is extreme in $\px$.
\end{corollary}

We briefly discuss here a result in Holevo \cite{Holevo} (see Theorem 2.1.2 therein),  which describes some significant differences that can arise when  dimension of the Hilbert space changes from finite to infinite.  Let $\mathcal{P}^0_\h(X)$ denote the set of all spectral measures, and let $\mathcal{P}^1_\h(X)$ denote the set of POVMs with commuting ranges. Note that $\mathcal{P}^0_\h(X)\subseteq \mathcal{P}^1_\h(X)$. Let $\Ext(\px)$ denote the set of extreme points of $\px$, and let $\co(S)$ denote the convex hull of a subset $S$ of $\px$. Holevo considers the following topology on $\px$ given by the convergence: a net $\mu_i$ converges to $\mu$ in $\px$ if $\tr(T\mu_i(A))\to \tr(T\mu(A))$, for all $A\in \ox$ and  trace class operators $T$ on $\h$ i.e. $\mu_i(A)\to \mu(A)$ in $\sigma$-weak topology (this is equivalent to saying that $\mu_i(A)\to \mu(A)$ in WOT for all $A\in\ox$). In Section \ref{povm on topological space}, we also consider a strictly weaker topology  that we define for POVMs on topological spaces  (see Definition \ref{definition of bw topology on povm}).

Let $n$ denote the cardinality of the set $X$. If $n=2$, then the extreme points of $\px$ are exactly the spectral measures (this case relates closely to the classical probability theory), as well as we have $\mathcal{P}^1_\h(X)=\overline{\co}(\mathcal{P}^0_\h(X))$. Note that this happens regardless of the dimension of the Hilbert space $\h$. On the other hand, when $n>2$ the situation becomes more complicated. To be precise, if $n>2$ then there are always some extreme points of $\px$ which are not spectral measures, and we have  $\mathcal{P}^1_\h(X)\subsetneq \overline{\co}(\mathcal{P}^0_\h(X))\subseteq\px$. Moreover, the latter inclusion is strict when $\dim\h<\infty$, while they are equal when $\dim\h=\infty$.

The scenario in the case of $\cst$-extreme points of $\px$ (see Definition \ref{definition of C*-extreme points}) is less complex. Indeed if $X$ is countable, or  if $\dim\h<\infty$ and $X$ is arbitrary, then $\cst$-extreme points of $\px$ are always spectral measures (see Theorem \ref{atomic $C^*$-extreme points are PVM} and Theorem \ref{POVMs on finite dimensional spaces are spectral} below). This is in stark contrast with extreme points case. Since spectral measures are more tractable objects and have classical significance, it seems very natural to study the theory of $\cst$-convexity of POVMs.

\subsection{Atomic and non-atomic POVMs}
One of the approaches that we take in this paper for exploring
$\cst$-extreme points  is via the decomposition of POVMs into atomic
and non-atomic POVMs and analysing them separately. So we recall
here the definitions and give some of their properties. These notions have been widely studied in classical measure theory. See \cite{Johnson} for a very general exposition.

\begin{definition}\label{definition of atomic and non-atomic POVM}
Let $\mu:\ox\to\bh$ be a POVM. A subset $ A\in\ox$ is called an
\emph{atom} for $\mu$ if $\mu( A)\neq0$ and whenever $ B\subseteq A$
in $\ox$,
\begin{align*}
    \text{either }\;\; \mu( B)=0\;\; \text{or}\;\;\mu( B)=\mu( A).
\end{align*}
 A POVM $\mu$ is called \emph{atomic} if every $ A\in\ox$ with $\mu( A)\neq0$ contains an atom. A POVM $\mu$ is called \emph{non-atomic} if it has no atom.
\end{definition}

We shall frequently make  use of the following remark, which is easy
to verify.
\begin{remark}
If $A$ is an atom for  a POVM $\mu$ then for any $B\subseteq A$ in $\ox$, either $\mu(B)=0$ or $B$ is an atom for $\mu$.
\end{remark}

It is a well-known fact that every finite (more generally
$\sigma$-finite) positive measure decomposes uniquely as a sum of an
atomic positive measure and a non-atomic positive measure. In a similar fashion, every
POVM decomposes uniquely  as a sum of an atomic POVM and a
non-atomic POVM (\cite{Mclaren Plosker Ram}, \cite{Davies}).
Although the proof in \cite{Mclaren Plosker Ram} (which itself is
inspired from the classical case) is for POVMs on  locally compact
Hausdorff spaces,  the same proof will work for general measurable
spaces (see the proof of Theorem \ref{thm:every povm decomposes as
direct sum of atomic and non atomic povm} below). We state it here.

\begin{theorem}\label{thm:every povm decomposes as a sum of atomic and non atomic povm}(Theorem 3.10, \cite{Mclaren Plosker Ram})
 Every POVM decomposes uniquely as a sum of an atomic POVM and a non-atomic POVM.
\end{theorem}

We end this section by making an useful observation on atoms of
POVMs which shall be frequently used  in the paper.

\begin{proposition}\label{mu is atomic iff pi is atomic}
Let $\mu:\ox\to\bh$ be a POVM with the minimal Naimark dilation
$(\pi,V,\hpi)$. Then a subset $ A\in\ox$ is an atom for $\mu$ if
and only if $ A$ is an atom for $\pi$. In particular, $\mu$ is
atomic (non-atomic) if and only if $\pi$ is atomic (non-atomic).
\end{proposition}
\begin{proof}
For any subset $ A\in\ox$,  $ A$ is an atom for $\mu$ if and  only
if $\mu(A)\neq 0$ and for each $ A'\subseteq A$ in $\ox$, we have either $\mu( A')=0$ or $\mu( A\setminus
A')=0$. Equivalently  $\pi(A)\neq 0$ and we have  either $\pi( A')=0$ or $\pi( A\setminus
A')=0$ from Proposition \ref{zero sets of mu and pi are same}, which
in turn is same as saying that $ A$ is an atom for $\pi$. The second assertion easily follows from the first.
\end{proof}

\section{Main Results on $C^*$-extreme Points}\label{main results on C*-extreme points}

As  mentioned earlier, $\px$ denotes the collection of all normalized
POVMs from $\ox$ to $\bh$.  We already saw that $\px$ is a convex
set and Theorem \ref{thm:extrme point criterion for POVM} gives an
abstract characterization of extreme points of $\px$. In the rest
of the paper, we  look into a non-commutative convexity structure of
$\px$, called quantum convexity or $\cst$-convexity. As said
earlier, the notion of $\cst$-convexity was  introduced in
\cite{Loebl Paulsen} for a subset of  $\bh$. In \cite{Farenick
Morenz}, it is generalized to the collection of unital completely
positive maps. Further one can see the definition of $C^*$-convexity
being modified and studied by \cite{Magajna} in different settings.
The notion has also been studied by Farenick et al. \cite{Farenick
Plosker Smith} for positive operator valued measures, which is our
main interest in this paper. Some general references on this topic
are \cite{Loebl Paulsen},  \cite{Farenick Morenz}, \cite{Farenick
Zhou}, \cite{Gregg}, \cite{Zhou}, \cite{Magajna} and \cite{Farenick
Plosker Smith}.

\begin{definition}\label{definition of C*-convex sets}
For any $\mu_i\in\px$ and $T_i\in\bh$, $1\leq i\leq  n$ with $\sum_{i=1}^n T_i^*T_i=I_\h$, a sum of the form
\begin{equation}\label{C^*sum}
 \mu(\cdot)=   \sum_{i=1}^nT_i^*\mu_i(\cdot)T_i
\end{equation}
 is called a \emph{$\cst$-convex combination} for $\mu$. The operators $T_i$'s here
  are called {\em $\cst$-coefficients}. When $T_i$'s are invertible,
  the sum in \eqref{C^*sum} is called a
  \emph{proper $\cst$-convex combination} for $\mu$.
  \end{definition}

 Observe that $\px$ is a {\em $\cst$-convex set} in the sense that it is closed
under $\cst$-convex   combinations. Now the following definition of $\cst$-extreme points is
the POVM analogue of the definition in \cite{Farenick Morenz} for
unital completely positive maps.

\begin{definition}\label{definition of C*-extreme points}
A normalized POVM $\mu:\ox\to\bh$ is called a \emph{$C^*$-extreme point} in
$\px$ if, whenever  $\sum_{i=1}^nT_i^*\mu_i(\cdot)T_i$  is a proper
$C^*$-convex combination for $\mu$, then each $\mu_i$ is unitarily
equivalent to $\mu$ i.e. there are unitary operators $U_i\in\bh$
such that $\mu_i(\cdot)=U_i^*\mu(\cdot)U_i $ for $1\leq i\leq n.$
\end{definition}

\subsection{Abstract characterizations of $C^*$-extreme points}
 Farenick and Zhou  \cite{Farenick Zhou} obtained a
 characterization of $\cst$-extreme points for unital completely positive maps.
  The same can be translated into the language of POVMs and one obtains a characterization for
   $\cst$-extreme points of $\px$.

As we are dealing with the more general case of  arbitrary
measurable spaces and also because we  are deliberately making
slight changes in the statements, we are providing the proof here
for completeness.

\begin{theorem}\label{thm:Farenick and Zhou characterization of $C^*$-extreme points}(Theorem 3.1, \cite{Farenick Zhou})
Let $\mu:\ox\to\bh$ be a normalized POVM with  the minimal Naimark
dilation $(\pi,V,\hpi)$. Then $\mu$ is a $C^*$-extreme point in
$\px$ if and only if for any positive operator $D\in\pi(\ox)'$ with
$V^*DV$ being invertible, there exists a co-isometry $U\in\pi(\ox)'$
(i.e. $UU^*=I_{\hpi})$ satisfying $U^*U\sqd=\sqd$ and an invertible operator
$S\in\bh$ such that $UD^{1/2}V=VS$.
\end{theorem}
\begin{proof}
First assume that $\mu$ is $C^*$-extreme in $\px$. Let
$D\in\pi(\ox)'$ be positive with $V^*DV$ invertible.  Choose
$\alpha>0$  such that $\|\alpha D\|< 1.$ This ensures that $I_{\h_\pi}-\alpha D$ is positive and invertible. Also $\|\alpha
V^*DV\|<1$ and hence $I_\h-\alpha V^*DV$ is positive and
invertible. Set
$$T_1=(\alpha V^*DV)^{1/2}~\text{ and }~T_2=(I_\h-\alpha V^*DV)^{1/2}.$$
Then both $T_1$ and $T_2$ are invertible and $T_1^*T_1+T_2^*T_2=I_\h$. Now we define POVMs $\mu_i:\ox\to\bh$, $i=1,2$ by
\begin{align}
    \mu_1(A)=T_1^{-1}\left(\alpha V^*D\pi(A)V\right)T_1^{-1} \text{ and }\mu_2(A)=T_2^{-1}\left(V^*
    (I_{\h_\pi}-\alpha D)\pi(A)V\right)T_2^{-1},
\end{align}
for all $A\in\ox$. It is clear that $\mu_i$ is a POVM and
$\mu_i(X)=I_\h$.  Also,
\begin{align*}
    T_1^*\mu_1(A)T_1+T_2^*\mu_2(A)T_2=V^*\pi(A)V=\mu (A)~~\mbox{ for all }A\in\ox.
\end{align*}
 Therefore since $\mu$ is $C^*$-extreme, there exists a  unitary $W\in\bh$ such that $\mu(\cdot)=W^*\mu_1(\cdot)W$. This  implies
$$  \mu(\cdot)=W^*T_1^{-1}(\alpha V^*D\pi(\cdot)V)T_1^{-1}W
    =\left(\sqrt{\alpha}W^*T_1^{-1}V^*\sqd\right)\pi(\cdot)\left(\sqrt{\alpha}\sqd VT_1^{-1}W\right)
    =V_1^*\pi(\cdot)V_1,$$
where $V_1=\sqrt{\alpha}\sqd VT_1^{-1}W\in\B(\h,\hpi)$. Note that
$V_1^*V_1=V_1^*\pi(X)V_1=\mu(X)=I_\h,$ and so $V_1$ is an isometry.
Now if we set
$$\K=[\pi\left(\ox\right)V_1\h]\subseteq\hpi,$$
then $\K$ is a reducing subspace for all $\pi(A)$,  $A\in \ox $. Also
$\ran(V_1)=\pi(X)V_1\h\subseteq\K,$ and if we think $V_1$ as an
operator from $\h$ into $\K$,  then  $(\pi(\cdot)_{|_\K},V_1,\K)$ is
the minimal Naimark dilation for $\mu$. Therefore, by the uniqueness
of  minimal dilation, there exists a unitary $U:\K\to\hpi$ satisfying
\begin{align*}
    UV_1=V\text{ and } \pi(A)U=U\pi(A)_{|_\K}~~\mbox{ for all }A\in\ox.
\end{align*}
Extend $U$ to the whole of $\hpi$ by assigning  it to be $0$ on
$\hpi\ominus\K$, which we  still denote
by $U$. Here $\hpi\ominus\K$ denotes the orthogonal complement of $\K$ in $\hpi$. It is immediate that
$$U^*U=P_\K~\text{ and }~UU^*=I_{\hpi}$$
where $P_\K$ is the projection of $\hpi$ onto $\K$, which is to say that  $U$ is a co-isometry. We also have
$$\pi(A)U=U\pi(A)~~\mbox{ for all } A\in\ox,$$
 and hence $U\in\piox'$. Set
$$S=\sqrt{\alpha}^{-1}W^*T_1\in\bh.$$
 Then $S$ is invertible and, since $UV_1=V$ and $W$ is a unitary, we get
\begin{align*}
    VS=UV_1S=\sqrt{\alpha} \sqrt{\alpha}^{-1}U\sqd V  T_1^{-1}WW^*T_1 =U\sqd V.
\end{align*}
The only remaining thing is to show  that $U^*U\sqd=\sqd$.  Since
$U^*U$ is a projection onto $\K$, this will follow once we show that
$\K=\cran(\sqd)$. Now using $\hpi=[\piox V\h]$ and invertibility
of $T_1$ and  $W$, we obtain
\begin{align*}
    \K=[\piox V_1\h]=[\piox\sqd V\left( \sqrt{\alpha}T_1^{-1}W\h\right)] =[\sqd \piox V\h]=\cran(\sqd).
\end{align*}
For the converse, assume that the given statement in `only if'
part  is true. Let $\mu=\sum_{i=1}^nT_i^*\mu_i(\cdot) T_i$ be a
proper $C^*$-convex combination. Fix any $i\in \{1,\ldots,n\}$.
Since $T_i^*\mu_i(\cdot) T_i\leq \mu ,$ it follows from
Radon-Nikodym type Theorem (Theorem \ref{thm:Radon-Nikodym type
theorem}) that there exists a positive operator $D_i\in\pi(\ox)'$
satisfying
$$T_i^*\mu_i(A)T_i=V^* D_i\pi(A) V~~\mbox{ for all }A\in\ox.$$
Then $V^*D_iV=T_i^*T_i$ and since $T_i$ is invertible, it follows that $V^*D_iV$ is invertible. Hence the hypothesis
ensures the existence of an operator $U_i\in\piox'$ satisfying
$U_i^*U_iD_i^{1/2}=D_i^{1/2}$ and an invertible
operator $S_i\in\bh$ such that $U_iD_i^{1/2}V=VS_i$. Thus,
\begin{align*}
 T_i^*\mu_i(\cdot)T_i
    &=V^*D_i\pi(\cdot)V=V^*D_i^{1/2}\pi(\cdot)D_i^{1/2}V=V^*D_i^{1/2}\pi(\cdot)U_i^*U_iD_i^{1/2}V=V^*D_i^{1/2}U_i^*\pi(\cdot) U_iD_i^{1/2}V\\
    &=\left(U_iD_i^{1/2}V\right)^*\pi(\cdot)\left(U_iD_i^{1/2}V\right)
    =(VS_i)^*\pi(\cdot)(V S_i)=S_i^*\left(V^*\pi(\cdot)V\right)S_i
    =S_i^*\mu(\cdot) S_i,
\end{align*}
which implies $\mu_i(\cdot)=T_i^{*^{-1}}S_i^*\mu(\cdot)
S_iT_i^{-1}=R_i^*\mu(\cdot) R_i,$ where $R_i=S_iT_i^{-1}$. It is
clear that $R_i$ is invertible and  since, $R_i^*R_i=\mu_i(X)=I_\h,$
it follows that $R_i$ is a unitary. This shows that $\mu_i$ is
unitarily equivalent to $\mu$, as required to conclude that $\mu$ is
a $C^*$-extreme point in $\px$.
\end{proof}

\begin{remark}
 In the statement of the theorem above, $U$ is a co-isometry. It is not
 clear at this point as to whether $U$ can be chosen to be a unitary, as claimed in (Theorem 3.1, \cite{Farenick Zhou}).
 Of course this is automatic if $\hpi $ is finite dimensional.
\end{remark}

The following is an immediate corollary of Theorem \ref{thm:Farenick and Zhou characterization of $C^*$-extreme points}. This can also be deduced from a result of Loebl and Paulsen (Proposition 26, \cite{Loebl Paulsen}), which says that projections are $\cst$-extreme points in the set of all positive contractions of $\bh$, although it needs a bit of effort.

\begin{corollary}\label{PVM is C*-extreme}
Every spectral measure is a $C^*$-extreme point in $\px$.
\end{corollary}
\begin{proof}
If $\mu $ is a spectral measure then the minimal
dilation for $\mu$ can be taken to be $(\mu , I_\h, \h )$. For positive $D\in \mu
(X)'$ with $D (=I_\h^*DI_\h)$ invertible, we can take $U=I_\h$ and $S=D^{1/2}$ to satisfy the
criterion.
\end{proof}

Zhou in his thesis  \cite{Zhou} gave another
characterization of $\cst$-extreme points for unital completely
positive maps.  The result translates to POVM case as follows. Again as there is a slight change in the statement, we provide the proof.

\begin{corollary}\label{Zhou Characterization of $C^*$-extreme points}
(Theorem 3.1.5, \cite{Zhou})
Let $\mu\in\px$. Then $\mu$  is $C^*$-extreme in $\px$ if and only if for any
POVM $\nu:\ox\to\bh$ with $\nu\leq \mu$ and $\nu(X)$ invertible,
there exists an invertible operator $S\in\bh$ such that
$\nu(A)=S^*\mu(A)S$ for all $A\in\ox$.
\end{corollary}
\begin{proof}
First assume that $\mu$ is a $C^*$-extreme point in $\px$. Let $\nu:\ox\to\bh$ be a POVM such that $\nu\leq\mu$ and $\nu(X)$ is invertible. Let $(\pi,V,\hpi)$ be the minimal Naimark dilation for $\mu$.  By Theorem \ref{thm:Radon-Nikodym type theorem}, there exists a positive operator $D\in\pi(\ox)'$ such that
$$\nu(A)=V^*D\pi(A)V ~~\text{for all } A\in\ox.$$
Since $V^*DV=\nu(X)$ and $\nu(X)$ is invertible, it follows that $V^*DV$ is invertible. Therefore, by Theorem \ref{thm:Farenick and Zhou characterization of $C^*$-extreme points} there exists a co-isometry  $U\in\pi(\ox)'$ satisfying $U^*U\sqd=\sqd$ and an invertible operator $S\in\bh$ such that $UD^{1/2}V=VS$. So for any $A\in\ox$, we get
\begin{align*}
    \nu(A)&=V^*D\pi(A)V=V^*D^{1/2}\pi(A)D^{1/2}V =V^*D^{1/2}\pi(A)U^*UD^{1/2}V
    =V^*D^{1/2}U^*\pi(A)UD^{1/2}V\\
    &=\left(UD^{1/2}V\right)^*\pi(A)\left(UD^{1/2}V\right)=(VS)^*\pi(A)(VS)=S^*\left(V^*\pi(A)V\right)S=S^*\mu(A)S.
\end{align*}
Conversely, assume the given statement in the `only if' part is true. Let
$\mu=\sum_{i=1}^nT_i^*\mu_i(\cdot)T_i$ be a proper $C^*$-convex combination. Then  $T_i^*\mu_i(\cdot)T_i\leq \mu(\cdot)$ for each $i$. Also, since
$T_i^*\mu_i(X)T_i=T_i^*T_i$
and $T_i$ is invertible, it follows that $T_i^*\mu_i(X)T_i$ is invertible. Hence by hypothesis, there exists an invertible operator $S_i\in\bh$ such that for all $A\in\ox$, we have
$$T_i^*\mu_i(A)T_i=S_i^*\mu(A)S_i$$
which when put differently yields $$\mu_i(A)=U_i^*\mu(A)U_i,$$
where $U_i=S_iT_i^{-1}$. But, since $U_i^*U_i=U_i^*\mu(X)U_i=\mu_i(X)=I_\h$ and $U_i$ is invertible, it follows that $U_i$ is a unitary.
This shows that $\mu_i$ is unitarily equivalent to $\mu$, as was required.
\end{proof}

We wish to mention that the condition of $\nu(X)$ being invertible
in the corollary above cannot be dropped.  The  original statement
(Theorem 3.1.5, \cite{Zhou}) is somewhat ambiguous about the
invertibility requirement in the characterization. But it is crucial
as the following example shows. Here $\T$ is the unit circle and $\mathcal{O}(\T)$ is the Borel $\sigma$-algebra of $\T$.

\begin{example}
Consider the normalized POVM $\mu:\mathcal{O}(\T)\to\B(\h^2)$ defined by
\begin{align*}
   \mu(A)=P_{\h^2}{M_{\chi_A}}_{|_{\h^2}}=T_{\chi_A}~~\text{for all }A\in \mathcal{O}(\T),
\end{align*}
where $\h^2$ is the Hardy space on $\T$ and $T_f=P_{\h^2}{M_f}_{|_{\h^2}}$ denotes the Toeplitz operator for any $f\in L^\infty$($=L^\infty(\T,l)$, where $l$ denotes the one-dimensional Lebesgue measure on $\T$). Here $\chi_A$ denotes the characteristic function for the subset $A$. It is known that $\mu$ is $\cst$-extreme in $\p_{\h^2}(\T)$ (see Example \ref{existence of a non-homomorphic C^*-extreme points on an uncountable metric space}).
Let $C\subseteq \T$ be a Borel subset such that  $l(C)\neq0$ and $l(\T\setminus C)\neq0$. Consider $\nu:\mathcal{O}(\T)\to\B(\h^2)$ defined by
\begin{align*}
    \nu(A)=\mu(A\cap C)=P_{\h^2}{M_{\chi_{(A\cap C)}}}_{|_{\h^2}}=T_{\chi_{(A\cap C)}}~~~\text{for all  }A\in\mathcal{O}(\T).
\end{align*}
It is  clear that $\nu$ is a POVM and  $\nu\leq \mu$. Also $\nu(\T)=T_{\chi_C}$ is not invertible. We claim that there is no operator $S\in\bh$ such that $\nu(\cdot)=S^*\mu(\cdot) S.$ Suppose this is not the case and $S$ is one such operator.
Note that $S^*S=\nu(\T)=T_{\chi_{C}}$. Since $l(C)$ and $l(\T\setminus C)$ are non-zero, we have $\chi_C\neq0$ and $\chi_{(\T\setminus C)}\neq0$ in $L^\infty$. It is then a  fact due to Coburn that $T_{\chi_C}$ and $T_{\chi_{(\T\setminus C)}}$ are  one-one operators (see Proposition 7.24, \cite{Douglas}) and hence $S^*S$ is one-one, which further implies that $S$ is one-one. Therefore, again as $T_{\chi_{(\T\setminus C)}}$ is one-one, it follows the operator $T_{\chi_{(\T\setminus C)}}S$ is one-one. But on the other hand, we have
\begin{align*}
    \left(T_{\chi_{(\T\setminus C)}}^{1/2}S\right)^*\left(T_{\chi_{(\T\setminus C)}}^{1/2}S\right)=S^*T_{\chi_{(\T\setminus C)}}S=S^*\mu(\T\setminus C)S=\nu(\T\setminus C)=0
\end{align*}
which implies
\begin{align*}
    T_{\chi_{(\T\setminus C)}}^{1/2}S=0
\end{align*}
and hence $T_{\chi_{(\T\setminus C)}}S=0$, leading us to a contradiction.
\end{example}

\subsection{$C^*$-extreme points with commutative ranges}
With these  two characterizations of $\cst$-extreme points at our
disposal, we are now ready to present the main results of this paper.
 Gregg \cite{Gregg} shows that if a POVM  $\mu $ is $C^*$-extreme
in $\px$ (for a compact Hausdorff space $X$) then  for any $A$ in $\ox$, the spectrum of $\mu (A)$ is
either contained in $\{ 0, 1\}$ (so that $\mu (A)$ is a projection)
 or it is whole of the interval $[0,1].$ Our main observation is
 that the second situation can be avoided in a variety of cases. The
 proof uses  straightforward Borel functional calculus, with a carefully chosen
 family of functions. These functions are necessarily discontinuous
 and so $C^*$-algebra setting and continuous functional calculus will not suffice.

\begin{theorem}\label{if mu(A) commutes with everything, then it is a projection}
Let $\mu$ be a  $C^*$-extreme point in $\px$. If $E\in\ox$ is such that $\mu(A)\mu(E)=\mu(E)\mu(A)$ for all $A\subseteq E$  in $\ox$, then $\mu(E)$ is a projection. In particular if $\mu(E)$ commutes with all $\mu(B)$ for $B\in \ox,$ then $\mu(E)$ is a projection.
\end{theorem}
\begin{proof}
The second assertion is immediate from the first. So assume the hypothesis in the first statement. We claim that $\sigma(\mu(E))\cap (r,s)=\emptyset$ for all $0< r<s<1$,  where $\sigma(\mu(E))$ denotes the spectrum of the operator $\mu(E)$.  As $\mu(E)$ is  a positive contraction,   it will follow that $\sigma(\mu(E))\subseteq\{0,1\}$, which in turn will imply that $\mu(E)$ is a projection. So fix $0<r< s< 1$, and define the map $f:=f_{r,s}:[0,1]\to [0,1]$ by
\begin{equation}\label{eq:expression of f_{r,s}}
    f_{r,s}(t)=\left\{\begin{array}{cc}
        1 & \mbox{if} ~t\notin [r,s], \\
         \frac{r}{1-r}\left(\frac{1}{t}-1\right)& \mbox{if}~ t\in [r,s].
    \end{array}\right.
\end{equation}
Clearly $f$ is continuous except at one point namely $s$, and hence
it  is a Borel measurable   function. So  for any operator $0\leq
T\leq I_\h$, it follows from  spectral
theory that $f(T)$ is a well defined bounded operator. Further we note for each $t\in[0,1]$, that
$$0<\alpha:=\left(\frac{r}{1-r}\right)\left(\frac{1-s}{s}\right)\leq f(t)\leq1$$ and consequently,
\begin{equation}\label{eq:inequality of f(T))}
\alpha I_\h\leq f(T)\leq I_\h.
\end{equation}
Now consider the map  $\nu:\ox\to\bh$ defined by
\begin{equation}
    \nu(B)=\mu(B\cap E)f(\mu(E))+\mu(B\setminus E)
\end{equation}
 for any $B\in\ox$.
 We show that $\nu$ is a POVM by observing the following:
\begin{itemize}
    \item For each $B\in\ox$, our hypothesis  says that  $\mu(B\cap E)$ and $\mu(E)$ commute and  it then implies from spectral theory that $\mu(B\cap E)$ commutes with $f(\mu(E))$. Therefore, as both $\mu(B\cap E)$ and $f(\mu(E))$ are positive operators, it follows that their product $\mu(E\cap B)f(\mu(E))$ is a positive operator, which amounts to saying  that $\nu(B)\geq 0$ in $\bh$.
    \item If $B_1,B_2,\ldots$ is a countable collection of mutually disjoint measurable  subsets of $X$ and  $B=\cup_n B_n$, then since $\mu$ is a POVM, we have in WOT convergence,
\begin{align*}
    \nu(\cup_n B_n)&=\mu((\cup_n B_n)\cap E) f(\mu(E))+\mu((\cup_n B_n)\setminus E)\\
    &=\mu(\cup_n (B_n\cap E))f(\mu(E))+\mu(\cup_n (B_n\setminus E))\\
    &=\sum_n[\mu(B_n\cap E)f(\mu(E))]+\sum_n\mu(B_n\setminus E)\\
    &=\sum_n\left[\mu(B_n\cap E)f(\mu(E))+\mu(B_n\setminus E)\right]\\
    &=\sum_n\nu(B_n).
\end{align*}
This shows that $\mu$ is countably additive, which  in particular
implies that the  function   $B\mapsto \langle h, \nu(B)k\rangle$ is
a complex  measure on $X$ for all $h,k\in\h$.
\end{itemize}
The  observations above imply that $\nu$ is a POVM. Further since $f(\mu(E))\leq I_{\h}$ from \eqref{eq:inequality of f(T))}, it follows for each $B\in \ox$, that
\begin{align*}
    \nu(B)=\mu(B\cap E)f(\mu(E))+\mu(B\setminus E)\leq \mu(B\cap E)+\mu(B\setminus E)=\mu(B)
\end{align*}
which is to say $\nu\leq \mu$. Also since $f(\mu(E))\geq \alpha I_\h$ from  \eqref{eq:inequality of f(T))}, and  $\mu(E)\leq I_\h$, we note that
\begin{align*}
    \nu(X)&=\mu( E)f(\mu( E))+\mu(X\setminus E)\\
    &\geq \alpha\mu( E)+\mu(X\setminus E)\\
    &=\alpha\mu( E)+I_\h-\mu( E)\\
    &=I_\h-(1-\alpha)\mu( E)\\
    &\geq I_\h-(1-\alpha)I_\h\\
    &=\alpha I_\h,
\end{align*}
which is equivalent to saying that $\nu(X)$ is invertible.  Therefore, as $\mu$ is a $C^*$-extreme point in $\px$, it follows from  Corollary \ref{Zhou Characterization of $C^*$-extreme points} that there exists an invertible operator $T\in \mathcal{B}({\mathcal{H}})$ satisfying the condition
\begin{align}\label{eq: expression of nu from zhou criterion}
    \nu(B)=T^*\mu(B)T ~~\mbox{ for all }B\in\ox.
\end{align}
 We note that $\nu(X)=T^*T=|T|^2$ and hence,
\begin{align}\label{eq:extrepresson for |T|}
     |T|=\nu(X)^{1/2}= \left[\mu(E)f(\mu(E))+I_\h-\mu(E)\right]^{1/2}
\end{align}
where $|T|$ denotes the square root of the positive operator $T^*T$. Set $S=\mu( E)$.
By taking $B=E$ in  \eqref{eq: expression of nu from zhou criterion}, we have
\begin{align*}
T^*ST=T^*\mu(E) T=\nu( E)=\mu( E)f(\mu( E))=Sf(S).
\end{align*}
Let $T=U|T|$ be the polar decomposition of $T$. Then  $U$ is a
unitary and $|T|$ is invertible, as  $T$ is invertible. Consequently,
\begin{align}\label{eq:U^*SU=g(S)}
    U^*SU=|T|^{-1}Sf(S)|T|^{-1}.
\end{align}
 Now let $g:[0,1]\to[0,1]$ be the map defined by
 \begin{align*}
 g(t)=\frac{tf(t)}{1-t+t f(t)}
 =\left\{\begin{array}{cc}
         t & \mbox{if}~t\notin [r,s],  \\
         r & \mbox{if} ~t\in [r,s].
    \end{array}\right.
\end{align*}
Then $g(S)$ is a well-defined bounded operator and we get
$$g(S)=Sf(S)[I_\h-S+Sf(S)]^{-1}.$$
Hence  \eqref{eq:extrepresson for |T|} and \eqref{eq:U^*SU=g(S)} yield
$$U^*SU=g(S).$$
Therefore by spectral mapping theorem (Theorem IX.8.11,
\cite{Conway}), spectrum of $S$ satisfies the following:
$$\sigma(S)=\sigma(U^*SU)=\sigma(g(S))\subseteq \text{essran}(g),
$$
where essran($g$) is the essential range of $g $ with respect to the spectral measure corresponding to the operator $S$.  But,
\begin{align*}
    \mbox{essran}(g)\subseteq \cran{(g)}\subseteq[0,r]\cup[s,1],
\end{align*}
which implies that $\sigma(S)\subseteq [0,r]\cup[s,1]$. This is same
as saying  $\sigma(S)\cap (r,s)=\emptyset$, which is what we
wanted to show.
\end{proof}

A direct application of Theorem \ref{if mu(A) commutes with everything, then it is a projection} is possible for $\cst$-extreme points with commutative range. We say
a POVM $\mu$ to be {\em commutative} if its range is commutative. It has been shown  \cite{Heinosaari Pellonpaa} that a commutative normalized  POVM
is an extreme point in $\px$ if and only if it is spectral. A
similar kind of result for $C^*$-extreme points  holds true
following the theorem above; if a $C^*$-extreme point $\mu$ in $\px$
is commutative, then it follows from Theorem \ref{if mu(A) commutes
with everything, then it is a projection} that $\mu(A)$ is
projection for all $A\in\ox$ and hence $\mu$ is spectral. Thus we have arrived at the following theorem.

\begin{theorem}\label{commutative C*-extreme points are PVM}
Let $\mu:\ox\to\bh$ be a commutative normalized POVM.
Then $\mu$ is $C^*$-extreme in $\px$ if and only if it is a spectral
measure.
\end{theorem}

\subsection{Atomic $C^*$-extreme points}

Theorem \ref{if mu(A) commutes with everything, then it is a
projection} is quite powerful. Here we have more applications of it.
We examine atomic $\cst$-extreme points. First consider the following
lemma. Recall our assumption that singletons are measurable subsets.

\begin{lemma}\label{atoms are projections}
Let $\mu$ be a $C^*$-extreme  point in $\px$. Then $\mu(E)$ is a
projection for every atom $E$ for $\mu$. In particular $\mu(\{x\})$
is a projection for all $x\in X$ and consequently  $\mu(A)$ is a
projection for every countable subset $ A$ of $X$.
\end{lemma}
\begin{proof}
If $E$ is an atom for $\mu$ then for each $B\subseteq E$  in $\ox$,
either $\mu(B)=0$ or $\mu(B)=\mu(E)$  and hence, $\mu(B)$ commutes
with $\mu(E)$. Therefore Theorem \ref{if mu(A) commutes with
everything, then it is a projection} is applicable and it follows
that $\mu(E)$ is a projection. This further implies that for each
$x\in X$, since either $\mu(\{x\})=0$ or $\{x\}$ is an atom for $\mu$,
$\mu(\{x\})$ is a projection.

Finally let $x,y\in X$ be two distinct points
and set $P=\mu(\{x\})$  and $Q=\mu(\{y\})$.
 Note that
\begin{align*}
    &P+Q=\mu(\{x\})+\mu(\{y\})=\mu(\{x,y\})\leq I_\h
\end{align*}
and hence  $P\leq I_\h-Q$. Because $P$ and $Q$ are projections as proved above, it follows that $P(I_\h-Q)=P$, which in turn yields
$$PQ=0.$$
In other words, $\mu(\{x\})$ and $\mu(\{y\})$ are mutually orthogonal projections for any two distinct points $x$ and $y$. Therefore, for any at most countable subset $A=\{x_1,x_2,\ldots\}$ of $X$, the collection $\{\mu(\{x_n\})\}$ consists of projections mutually orthogonal to one another and  since
\begin{align*}
    \mu(A)=\sum_{n}\mu(\{x_n\})\quad\quad(\text{in WOT}),
\end{align*}
we conclude that $\mu(A)$ is a projection.
\end{proof}

The POVMs on finite sets  have been  natural setting for many
applications in quantum theory.  Several researchers have looked
into the convexity structure in this set up and the structure of
extreme points is very well studied. They are not always spectral
measures.  When it comes to $\cst$-convexity, it is shown in
\cite{Farenick Plosker Smith}  that only spectral measures are
$C^*$-extreme when $\h$ is finite dimensional. Here we show that it
is true in full generality.

Following the results above, we now give a characterization  of all
 atomic $C^*$-extreme points in $\px$. This in particular
characterizes all $\cst$-extreme points in $\px$ whenever $X$ is
finite.

\begin{theorem}\label{atomic $C^*$-extreme points are PVM}
An atomic normalized POVM $\mu$ on a measurable space $X$ is a $C^*$-extreme point in
$\px$ if and only if $\mu$ is  spectral. In particular, if $X$ is a
countable measurable space then any $C^*$-extreme point of $\px$ is
spectral.
\end{theorem}
\begin{proof}
We have seen that  spectral measures are always  $\cst$-extreme. Conversely,
assume that $\mu$ is $\cst$-extreme in $\px$. Let $\{B_i\}_{i\in I}$
be a maximal collection of mutually disjoint atoms for $\mu$, which exists thanks to
Zorn's lemma. Then  Lemma \ref{atoms are projections}  says
that $\mu(B_i)$ is a projection for each $i\in I$. Also, since $B_i$'s
are mutually disjoint, it follows from Proposition
\ref{prop:projection commutes with everything} that
$\{\mu(B_i):i\in I\}$ is a collection of  mutually orthogonal
projections. Hence, as $\h$ is separable, we conclude that $I$ is
countable. Further for any $A\in\ox$, we have
\begin{align}\label{eq:expression of mu(A) in terms of B_i}
    \mu(A)=\sum_{i\in I}\mu(A\cap B_i),
\end{align}
otherwise, we would get  $\mu(A\setminus(\cup_{i}(A\cap B_i)))\neq0$ and since $\mu$ is atomic, there is an atom, say $A_1\subseteq
A\setminus(\cup_{i}(A\cap B_i))$ for $\mu$. But then $\{B_i\}_{i\in
I}\cup\{A_1\}$ is a  collection of mutually disjoint atoms for $\mu$, which
violates the maximality of the collection $\{B_i\}_{i\in I}$. Similarly note from Lemma \ref{atoms are projections} that for
any $A\in \ox$, since either $\mu(A\cap B_i)=0$ or $A\cap B_i$ is an atom for $\mu$,
the collection $\{\mu(A\cap B_i)\}_{i\in I}$ consists of mutually orthogonal
projections. Consequently it follows from equation
\eqref{eq:expression of mu(A) in terms of B_i}, that $\mu(A)$ is a
projection. This proves that $\mu$ is  spectral. Since any POVM on a
countable measurable space is atomic, the second assertion follows.
\end{proof}

Since all spectral measures are also extreme (in the usual sense),
we have the following corollary. Note that the same is always true
for general measurable spaces, whenever $\h$ is a finite dimensional
Hilbert space (\cite{Farenick Plosker Smith}).

\begin{corollary}\label{C*-extrme implies extreme for countable stes}
If $X$ is a countable (in particular, finite) measurable space, then every $C^*$-extreme point in $\px$ is  extreme.
\end{corollary}

\subsection{The case of finite dimensional Hilbert space}
We end this section by recording the case of finite dimensional
Hilbert spaces and general measurable spaces. This set up has
been widely studied by several researchers. We recall that it is proved
in \cite{Farenick Plosker Smith} for a compact Hausdorff space $X$
and a finite dimensional $\h$, that every $\cst$-extreme point in $\px$ is
spectral. We extend this result to full generality using Theorem
\ref{atomic $C^*$-extreme points are PVM}.

\begin{theorem}\label{POVMs on finite dimensional spaces are spectral}
Let $\h$ be a finite dimensional Hilbert space and $X$ a measurable space. Then any $\cst$-extreme point in $\px$ is spectral.
\end{theorem}

\begin{proof}
Firstly, finite dimensionality of $\h$ ensures that every
$\cst$-extreme point in $\px$ is also extreme (Proposition 2.1,
\cite{Farenick Plosker Smith}). Now we show that every extreme point
in $\px$ is atomic (see Lemma 2, \cite{Chiribella Ariano Schlingemann} for
topological spaces) as follows: if $\mu$ is extreme in $\px$ and
$(\pi,V,\hpi)$ is the minimal Naimark dilation for $\mu$, then the
map $$D\mapsto V^*DV$$ from $\pi(\ox)'$ to $\bh$ is one-to-one by
Theorem \ref{thm:extrme point criterion for POVM}.
Since $\h$
is finite dimensional, $\bh$ is a finite dimensional algebra and hence $\pi(\ox)'$ is a finite-dimensional algebra.
Therefore, since $\pi(\ox)\subseteq\pi(\ox)'$ and $\hpi=[\piox
V\h]$, it follows that $\hpi$ is also finite-dimensional.

Consequently  $\{\pi (A): A\in \ox \}$ is a commuting family of
projections on a finite dimensional Hilbert space $\hpi $ and hence
it is a finite set. This implies that $\pi $ is atomic. Then by
Proposition \ref{mu is atomic iff pi is atomic},  $\mu$ is also
atomic. Thus we have shown that every $\cst$-extreme point in $\px$
is atomic. The proof is complete in view of Theorem \ref{atomic
$C^*$-extreme points are PVM}.
\end{proof}

\begin{remark}
In the theorem above, we noticed that any spectral measure acting on a finite dimensional Hilbert space is atomic.
\end{remark}

\section{Mutually Singular POVMs}\label{mutually singular povm}

\subsection{Mutual singularity}
The notion of mutual singularity of positive measures is very
familiar from classical measure theory. We consider the similar
notion of mutually singular POVMs.  Our main aim here is to discuss
the behaviour of $\cst$-extremity for direct sums of mutually
singular POVMs. This helps us in  characterization of $\cst$-extreme
points, as we show that every $\cst$-extreme POVM can be  decomposed
into a direct sum of an atomic and a non-atomic normalized POVM.

\begin{definition}\label{definition of mutual singularity}
Let $\h _1, \h _2$ be  Hilbert spaces and $X$ a measurable space. Two POVMs
$\mu_i:\ox\to\mathcal{B}(\h_i)$, $i=1,2,$  are called \emph{mutually
singular}, denoted $\mu_1\perp\mu_2,$ if there exist disjoint
measurable subsets $X_1$ and $X_2$ of $X$ such that $\mu_i(
A)=\mu_i( A\cap X_i)$ for all $ A\in\ox$.
\end{definition}

The following proposition is a direct consequence of the classical case. It is a well-known fact that an atomic finite positive measure is always
mutually singular to a non-atomic  positive measure. We use it below.

\begin{proposition}\label{atomic and non atomic povms are mutually singular}
Let $\mu_i:\ox\to\B(\h_i)$, $i=1,2$ be two POVMs such that $\mu_1$
is atomic and $\mu_2$ is non-atomic. Then they are mutually
singular.
\end{proposition}
\begin{proof}
Consider strictly positive density operators  $S_i$ on $\h_i $ such
that $T\mapsto  \tr (S_iT)$  ($\tr$ denotes trace) are faithful
normal states on $\B (\h_i)$ for $i=1,2.$  Then
$\lambda_i:\ox\to[0,\infty)$ defined by
\begin{align*}
    \lambda_i(A)=\tr(\mu_i(A)S_i)~~~\text{for all }A\in \ox,
\end{align*}
are positive measures which, for any $A\in\ox$ satisfy
\begin{align}\label{eq:lambda and mu has same zero measure sets}
    \mu_i(A)=0 \text{ if and only if }\lambda_i(A)=0.
\end{align}
This in particular implies that $\lambda_1$ is atomic and $\lambda_2$ is non-atomic. Therefore, as mentioned above,
 $\lambda_1$ is mutually singular to
$\lambda_2$ (Theorem
2.5, \cite{Johnson}). This in turn implies due to  \eqref{eq:lambda and mu
has same zero measure sets} that $\mu_1$ is mutually singular to
$\mu_2$.
\end{proof}

\subsection{Disjoint spectral measures}
Inspired by the notion of disjointness for representations of
$\cst$-algebras (see \cite{Arveson1},\cite{Arveson2}), we introduce
a similar notion for spectral measures. We do not know whether this concept has
been studied before.  We establish here that singularity and
disjointness of spectral measures are in fact same.

Let $\pi:\ox\to\mathcal{B}(\hpi)$ be a spectral measure and let $\h$ be
  a closed subspace of $\hpi$ such that $\h$ is invariant
  (and hence reducing) under $\pi( A)$ for all $ A\in\ox$. Then  the mapping
   $ A\mapsto\pi( A)_{|_{\h}}$ gives rise to another spectral measure from $\ox$  to
    $\bh$, and is called a \emph{sub-spectral measure} of $\pi$.

\begin{definition}
Two spectral measures $\pi_i:\ox\to\mathcal{B}(\h_{\pi_i})$, $i=1,2$ are called \emph{disjoint} if no
non-zero sub-spectral measure of $\pi_1$ is unitarily equivalent  to any sub-spectral measure of $\pi_2$.
\end{definition}

Let $\lambda:\ox\to [0,\infty]$ be a $\sigma$-finite measure such that $L^2(\lambda)$ is a separable Hilbert space.  Consider the map  $\pi^\lambda:\ox\to\mathcal{B}(L^2(\lambda))$  defined by
\begin{align}
    \pi^\lambda(A)=M_{\chi_A}~~\text{for all }A\in\ox,
\end{align}
where $M_{\chi_A}$ is the multiplication operator by the characteristic function $\chi_A$.
It is straightforward to verify that $\pi^\lambda$ is a spectral
measure. Also $\pi^\lambda(A)=0$ if and only if $\lambda(A)=0$ for
any $A\in\ox$. Such spectral measures are known as canonical
spectral measures.

We first prove that the notion of singularity and
disjointness are same in the case of canonical spectral measures. The proof here follows the same
technique which is  usually employed for representations (see Theorem 2.2.2,
\cite{Arveson2}).

\begin{lemma}\label{singulaity and disjointness of pi lambda}
 Let $\lambda_1$ and $\lambda_2$ be two $\sigma$-finite positive measures on $X$. Then $\lambda_1$ is
 mutually  singular  to $\lambda_2$ if and only if $\pi^{\lambda_1}$ and $\pi^{\lambda_2}$ are disjoint.
\end{lemma}
\begin{proof}
 Let $\pi^{\lambda_1}$ and $\pi^{\lambda_2}$ be disjoint spectral measures. Assume to the contrary that $\lambda_1$ and $\lambda_2$ are not mutually singular.
Then by Lebesgue decomposition theorem, there is a non-zero $\sigma$-finite positive
measure, say $\lambda$, such that $\lambda$ is absolutely continuous with respect
to both $\lambda_1$ and $\lambda_2$. Using Radon-Nikodym derivative $\frac{d\lambda}{d\lambda_i}$ of $\lambda$ with respect to $\lambda_i$,
it is not hard to see that
$\pi^\lambda$ is unitarily equivalent to $\pi^{\lambda_i}(\cdot)_{|_{\K_i}}$, where $K_i=\ran(\pi^{\lambda_i}(C_i))$ and $C_i=\{x\in X; \frac{d\lambda}{d\lambda_i}(x)>0\}$. It is clear that since $\lambda_i(C_i)\neq 0$,  we have $K_i\neq0$ which  contradicts
 disjointness of $\pi^{\lambda_1}$ and $\pi^{\lambda_2}$. The proof of the converse is contained in the
next theorem.
\end{proof}

We  use  the familiar notion of direct sum in the next theorem and in
subsequent results. The {\em direct sum} of a collection
$\{\mu_i:\ox\to\B(\h_i)\}_{i\in I}$ of POVMs is  the map
$\oplus_i\mu_i:\ox\to \B(\oplus_i\h_i)$ defined by
\begin{align}
    (\oplus_i\mu_i)(A)=\oplus_i\mu_i(A)~~\text{for all }A\in\ox.
\end{align}
It is immediate that  $\oplus_i\mu_i$ is a POVM.
Further it  is normalized if and only if each $\mu_i$
 is normalized. Also $\oplus_i\mu_i$ is a spectral measure if and only if each $\mu_i$ is a spectral measure.
Similar to an equivalent criterion for disjointness of
representations (Proposition 2.1.4, \cite{Arveson2}), we have the
following result for spectral measures. This also shows that the
notions of singularity and disjointness are same.

\begin{theorem}\label{prop:equiavelnt criterion for disjointness}
Let $\pi_i:\ox\to\mathcal{B}(\h_{\pi_i}), i=1,2$ be two  spectral measures. Then the following are equivalent:
\begin{enumerate}
\item $\pi_1$ and $\pi_2$ are mutually singular.
    \item $\pi_1$ is disjoint to $\pi_2$.
\item If for  $T\in\mathcal{B}(\h_{\pi_1},\h_{\pi_2})$,  $T\pi_1(A)=\pi_2(A)T$ for all $A\in\ox$, then $T=0$.
\end{enumerate}
\end{theorem}
\begin{proof}
$(1)\implies(3)$: Let $C_1$ and $C_2$ be disjoint measurable subsets such that
$\pi_i(A)=\pi_i(A\cap C_i)$ for all $A\in \ox$ and $i=1,2$. Let $T\in \B(\h_{\pi_1},\h_{\pi_2})$ be such that $T\pi_1(A)=\pi_2(A)T$ for all $A\in \ox$. Then, since $\pi_1(C_1)=I_{\h_{\pi_1}}$ and
$\pi_2(C_1)=0$, it follows that
\begin{align*}
    T=T\pi_1(C_1)=\pi_2(C_1)T=0.
\end{align*}

 $(3)\implies(2)$: if $\pi_1$ and $\pi_2$ are not disjoint, then there are non-zero closed subspaces $\K_i$ of $\h_{\pi_i}$ invariant under $\pi_i(A)$ for all $A\in\ox$, and a unitary $U:\K_1\to\K_2$ such that
 $$U\pi_1(A)_{|_{\K_1}}=\pi_2(A)_{|_{\K_2}}U~~\mbox{ for all }A\in\ox.$$
 Extend $U$ to $\h_{\pi_1}$ by assigning $0$ on $\h_{\pi_1}\ominus \K_1$, and call it $\tilde{U}$. Then it is immediate that $\tilde{U}\neq0$ and
  $\tilde{U}\pi_1(A)=\pi_2(A)\tilde{U}$ for all $A\in\ox$, violating the condition in  part (3).

$(2)\implies (1)$: Let $\pi_1$ and $\pi_2$ be disjoint.  By
Hahn-Hellinger Theorem (Theorem 7.6, \cite{KRP2}), there exists a
collection, say $\{\lambda^i_n\}_{n\in\N\cup\{\infty\}}$, of $\sigma$-finite
positive measures (possibly zero measures) mutually singular to one
another such that, upto unitary equivalence, we have
\begin{align*}
\pi_i=\bigoplus_{n\in\N\cup\{\infty\}}n\cdot\pi^{\lambda^i_n}
\end{align*} for  $i=1,2$. Here
$n\cdot \pi^{\lambda_n^i}$ denotes the direct sum of n copies of
$\pi^{\lambda_n^i}$ (when $n=\infty$, the direct sum is countably
infinite). Because $\pi_1$ and $\pi_2$ are disjoint, each
$\pi^{\lambda_n^1}$ must be  disjoint to $\pi^{\lambda_m^2}$ for
$m,n\in\N\cup\{\infty\}$. It then  follows from Lemma
\ref{singulaity and disjointness of pi lambda} that
$\lambda_n^1$ is mutually singular to $\lambda_m^2$ as positive measures.
Therefore for each $n,m$,  there exist
measurable  subsets $X^1_{nm}$ and $X^2_{nm}$ satisfying $X_{nm}^1\cap
X_{nm}^2=\emptyset$ and
\begin{align*}
\lambda_n^1( A)=\lambda_n^1( A\cap X^1_{nm})\;\;\text{ and
}\;\;\lambda_m^2( A)=\lambda_m^2( A\cap X^2_{nm}),
\end{align*}
for all $ A\in\ox$. Set
\begin{align*}
    X^1=\cup_n\cap_m X^1_{nm} \text{ and } X^2=\cup_m\cap_nX^2_{nm}.
\end{align*}
Then by usual set theory rules:
$$
    X^1\cap X^2=\left(\cup_n\cap_m X^1_{nm}\right)\cap\left(\cup_k \cap_l X^2_{lk}\right)
    =\cup_n\cup_k\left[\left(\cap_m X^1_{nm}\right)\cap \left(\cap_lX^2_{lk}\right)\right]
    \subseteq \cup_n\cup_k \left(X^1_{nk}\cap
    X^2_{nk}\right)=\emptyset, $$
by using $X^1_{nk}\cap X^2_{nk}=\emptyset.$ Further for any $A\in\ox$
and fixed $n$, since $\lambda_n^1(A\cap X^1_{nm})=\lambda_n^1(A)$
for all $m$, we have
$$\lambda_n^1(A)\geq\lambda_n^1(A\cap X^1)\geq \lambda_n^1(\cap_m\left( A\cap X^1_{nm})\right)=\lim_l\lambda_n^1\left(\cap_{m=1}^l(A\cap X^1_{nm})\right)=\lambda_n^1(A),$$
where limit  is taken in WOT. This implies $\lambda_n^1(A\cap
X^1)=\lambda_n^1(A).$ Similarly, we get $\lambda_m^2(A\cap
X^2)=\lambda_m^2(A)$ for each $m$. Put differently, we obtain
$\pi^{\lambda_n^i}(A\cap X^i)=\pi^{\lambda_n^i}(A),$ which further
implies that
\begin{align*}
    \pi_i(A\cap X^i)=\bigoplus_{n\in\N\cup\{\infty\}}n\cdot\pi^{\lambda_n^i}(A\cap X^i)=\bigoplus_{n\in\N\cup\{\infty\}}n\cdot\pi^{\lambda_n^i}(A)=\pi_i(A),
\end{align*}
for each $A\in\ox$ and $i=1,2$. Since $X^1$ and $X^2$ are disjoint,
we conclude that $\pi_1$ is mutually  singular to $\pi_2$.
\end{proof}

\begin{remark}\label{mutual singularity implies intertwining operators are zero}
In Theorem \ref{prop:equiavelnt criterion for disjointness}, we assumed that the spectral measures act on separable Hilbert spaces. But the implication (1)$\implies$(3) is true even for non-separable Hilbert spaces and the proof is similar. To see this, let $\pi_i:\ox\to\B(\K_i)$, $i=1,2$ be two mutually singular spectral measures  concentrated on measurable subsets $X_i$ with $X_1\cap X_2=\emptyset$. Here $\K_i$ need not be separable. Let $T\in \mathcal{B}(\K_1,\K_2)$ be such that $T\pi_1(A)=\pi_2(A)T$ for all $A\in\ox$. Then, since $\pi_1(X_1)=\pi_1(X)=I_{\K_1}$ and $\pi_2(X_1)=0$, we have $T=T\pi_1(X_1)=\pi_2(X_1)T=0$. We use this fact in the next theorem.
\end{remark}

\subsection{Direct sums and  $\cst$-extreme points}

We explore the properties of  being $\cst$-extreme or extreme  for
direct sums of mutually singular POVMs. Generally it is enough to
look at individual components to obtain the same property for direct
sums.

\begin{theorem}\label{direct sum of C*extreme points with disjoint support is $C^*$-extreme if and only if each component is $C^*$-extreme}
Let $\{\mu_i:\ox\to\B(\h_i)\}_{i\in I}$ be a countable collection of normalized POVMs  for some indexing set $I$
such that $\mu_i$ and $\mu_j$ are mutually  singular for  $i\neq j$
in $I$.  Then $\mu=\oplus_{i}\mu_i$ is $C^*$-extreme (extreme)  in
$\mathcal{P}_{\oplus_{i}\h_i}(X)$ if and only if each $\mu_i$ is
$C^*$-extreme (extreme)  in $\mathcal{P}_{\h_i}(X)$.
\end{theorem}
\begin{proof}
For each $i\in I$,  let $(\pi_i, V_i,\h_{\pi_i})$ be the minimal
Naimark dilation for $\mu_i$. Set $\h=\oplus_i\h_i$ and
$\hpi=\oplus_i \h_{\pi_i}$ and let
$\pi=\oplus_{i}\pi_i:\ox\to\mathcal{B}(\hpi)$ and $V=\oplus_i
V_i:\h\to\hpi$. Clearly $\pi$ is a spectral measure and $V$ is an isometry. It is straightforward to check that
\begin{align*}
    [\pi(\ox)V\h]=\hpi \text{ and }\mu(A)=V^*\pi(A)V~~\mbox{for all }A\in\ox.
\end{align*}
This implies that $(\pi,V,\hpi)$ is the minimal Naimark
dilation for $\mu$. Also for $i\neq j$ in $I$, since $\mu_i$ is mutually singular to $\mu_j$, it follows from Proposition \ref{zero sets of mu and pi are same} that
 $\pi_i$ is mutually singular to $\pi_j$.
Now we claim that
\begin{align*}
\pi(\ox)'=\oplus_i\pi_i(\ox)'=\left\{\oplus_iS_i;~S_i\in\pi_i(\ox)'\right\}.
\end{align*}
Let $S\in\pi(\ox)'\subseteq\mathcal{B}(\oplus_i\h_{\pi_i})$. Then $S=[S_{ij}]$ for some $S_{ij}\in\mathcal{B}(\h_{\pi_j},\h_{\pi_i})$. For any $A\in\ox$, therefore
\begin{align*}
&[S_{ij}]\left(\oplus_i\pi_i(A)\right)=\left(\oplus_i\pi_i(A)\right[S_{ij}]\end{align*}
that is
\begin{align*}
[S_{ij}\pi_j(A)]=[\pi_i(A)S_{ij}]
\end{align*}
and hence
$$S_{ij}\pi_j(A)=\pi_i(A)S_{ij}~~\text{for all } i,j\in I.$$
In particular, this says that $S_{ii}\in \pi_i(\ox)'$ for all $i\in I$. Also since $\pi_i$ and $\pi_j$ are mutually singular
for $i\neq j$, it follows from Remark \ref{mutual singularity implies intertwining operators are zero} that
$$S_{ij}=0~~\text{ for }i\neq j.$$
Thus
$$S=[S_{ij}]=\oplus_i S_{ii}\in\oplus_i\pi_i(\ox)'.$$
This proves that $\pi(\ox)'\subseteq\oplus_i\pi_i(\ox)'$. The other inclusion of our claim is obvious.

 In order to prove the equivalent assertions of $C^*$-extremity, we use the claim above and the necessary and sufficient criterion of Theorem \ref{thm:Farenick and Zhou characterization of $C^*$-extreme points} throughout the proof without always mentioning them. First assume  that $\mu$ is $C^*$-extreme in $\px$. Fix $j \in I$ and let  $D_j\in\pi_j(\ox)'$ be positive such that $V_j^*D_jV_j$ is invertible. Define
 $$D=\oplus_i D_i$$
 by assigning $D_i=I_{\h_{\pi_i}}$ for $i\neq j$.
It is clear that $D$ is positive and $D\in\pi(\ox)'$. Since
$$V^*DV=\oplus_iV_i^*D_{i}V_i$$
and $V_i^*D_iV_i$ is invertible for all $i$, it follows that $V^*DV$
is invertible. Therefore, as $\mu$ is $\cst$-extreme in $\px$, we get a co-isometry $U\in\pi(\ox)'$ with $U^*U\sqd =\sqd$ and an invertible operator $T\in \bh$ such that $U\sqd V=VT$. Then $T=[T_{ij}]$ for some $T_{ij}\in\mathcal{B}(\h_j,\h_i)$ and   $U=\oplus_iU_i$ for $U_i\in\pi_i(\ox)'$. Since $U$ is a co-isometry, each $U_i$ is a co-isometry. Also, since
\begin{align*}
    \oplus_iU_i^*U_i D_i^{1/2}=\left(\oplus_iU_i^*\right)\left(\oplus_iU_i\right)\left(\oplus_iD_i^{1/2}\right)=U^*U\sqd=\sqd=\oplus_i D_i^{1/2},
\end{align*}
it follows that
$$U_i^*U_iD_i^{1/2}=D_i^{1/2}~~\text{for all }i.$$
Further, since
\begin{align}\label{eq:UDV}
    \oplus_iU_iD_i^{1/2}V_i=U\sqd V=VT=(\oplus_iV_i)[T_{ij}]=[V_iT_{ij}],
\end{align}
it follows for $i\ne j$ that,   $V_iT_{ij}=0$ and hence $T_{ij}=V_i^*V_iT_{ij}=0$.  This amounts to saying that $T=\oplus_iT_{ii}$ and its invertibility, in particular,  implies that  $T_{jj}$ is invertible in $\mathcal{B}(\h_j)$. Also  \eqref{eq:UDV} yields
 \begin{align*}
     U_jD_j^{1/2}V_j=V_jT_{jj}.
 \end{align*}
 As $U_j$ is a co-isometry in $\pi_j(\ox)'$ satisfying $U_j^*U_jD_j^{1/2}=D_j^{1/2}$ and $T_{jj}$ is invertible in $\B(\h_j)$ such that $U_jD_j^{1/2}V_j=V_jT_{jj}$, we conclude that
 $\mu_j$ is $C^*$-extreme in $\p_{\h_j}(X)$.

Conversely, assume that each $\mu_i$ is $C^*$-extreme in $\p_{\h_i}(X)$. Let $D\in \pi(\ox)'$ be positive such that $V^*DV$ is invertible. Then $D$ is of the form $\oplus_iD_i$  for some $D_i\in\pi_i(\ox)'$. Clearly each $D_i$ is positive. Since $V^*DV$ is invertible and  $V^*DV=\oplus_iV_i^*D_iV_i$, it follows that $V_i^*D_iV_i$ is invertible for all $i\in I$. Again,
as $\mu_i$  is $C^*$-extreme in $\p_{\h_i}(X)$,  we obtain a co-isometry  $U_i\in \pi_i(\ox)'$ with $U_i^*U_iD_i^{1/2}=D_i^{1/2}$ and an invertible operator $T_i\in\mathcal{B}(\h_i)$ such that $U_iD_i^{1/2}V_i=V_iT_i$. Set
$$U=\oplus_iU_i\; \text{  and }\;T=\oplus_iT_i.$$
Then $U\in\piox'$ and $U$ is a co-isometry, as each $U_i$ is a co-isometry. Likewise $T$ is invertible in $\bh$, since each $T_i$  is invertible. Further we note that
\begin{align*}
    U^*U\sqd=\oplus_iU_i^*U_iD_i^{1/2}=\oplus_iD_i^{1/2}=\sqd.
\end{align*}
Similarly we get
$$U\sqd V=\oplus_iU_iD_i^{1/2}V_i=\oplus_iV_iT_i= VT.$$
Thus we conclude that $\mu$ is $C^*$-extreme in $\px$.

The case of equivalent assertions of extremity can be proved similarly, using the claim above and Theorem \ref{thm:extrme point criterion for POVM}. Assume that $\mu$ is extreme in $\px$. Fix $j\in I$ and let $D_j\in\pi_j(\ox)'$ be such that $V_j^*D_jV_j=0$. Define
$$D=\oplus_iD_i$$
by assigning $D_i=0$ for $i\neq j$. Clearly then  $V^*DV=0$. Hence, as $\mu$ is extreme in $\px$, it follows that $D=0$, which in particular implies $D_j=0$. This proves that $\mu_j$ is extreme in $\p_{\h_j}(X)$.

For the converse, assume that each $\mu_i$ is extreme in $\p_{\h_i}(X)$. Let $D\in\pi(\ox)'$ be such that $V^*DV=0$. Again by the claim above, we have $D=\oplus_i D_i$ for some $D_i\in\pi_i(\ox)'$. Also the expression $V^*DV=0$ implies
$$V_i^*D_iV_i=0~~\text{ for each }i.$$
 Hence as $\mu_i$ is extreme, it follows that $D_i=0$ for each $i$, which  in turn shows $D=0$. Thus we conclude that $\mu$ is extreme in $\px$. The proof is now complete.
\end{proof}

The following corollary is just an explicit description of the theorem above.

\begin{corollary}
Let $\mu\in\px$ and let $\{B_i\}$ be a collection of disjoint measurable  subsets such that $X=\cup_i B_i$ and $\mu(B_i)$ is a projection for each $i$. Let $\h_i=\ran(\mu(B_i))$ and  define $\mu_i:\ox\to\mathcal{B}(\h_i)$ by $\mu_i( A)=\mu(B_i\cap A)_{|_{\h_i}}$ for all $A\in \ox$. Then $\mu$ is  $C^*$-extreme (extreme) in $\px$ if and only if each $\mu_i$ is $C^*$-extreme (extreme) in $\mathcal{P}_{\h_i}(X)$.
\end{corollary}
\begin{proof}
Let $(\pi, V,\hpi)$ be the minimal Naimark dilation for $\mu$.  Since
$\mu(B_i)$ is a projection for each $i$ and $B_i$'s are  mutually
disjoint, it follows from Proposition \ref{prop:projection commutes with everything} that $\mu(B_i)$'s are mutually orthogonal projections. Also each $\h_i$ is a reducing subspace for all $\mu(A), A\in \ox$ by Proposition \ref{prop:projection commutes with everything} and hence $\mu_i$ is a well-defined normalized POVM. Further,
since $X=\cup_i B_i$, we have  $\h=\oplus_i \h_i$ and
$\mu=\oplus_i\mu_i$. The assertions now are  direct consequence of
Theorem \ref{direct sum of C*extreme points with disjoint support is
$C^*$-extreme if and only if each component is $C^*$-extreme}.
\end{proof}

As we mentioned earlier in Theorem \ref{thm:every povm decomposes as
a sum of atomic and non atomic povm}, every POVM decomposes uniquely
as a sum of atomic and non-atomic POVMs. Additionally if $\mu$ is
$C^*$-extreme then we show that this decomposition
can be made into  a  direct sum of atomic and non-atomic POVMs such that each of
 the components is $C^*$-extreme. The following theorem effectively
   provides a proof of Theorem \ref{thm:every povm decomposes as a sum of atomic and non atomic povm} and then
   discusses its role in identifying $C^*$-extreme POVMs.
   The proof here follows almost the same procedure which can be found in \cite{Mclaren Plosker Ram}, \cite{Johnson}.

\begin{theorem}\label{thm:every povm decomposes as direct sum of atomic and non atomic povm}
Let $\mu$ be a $\cst$-extreme point in  $\px$.  Then $\mu=\mu_{1}\oplus\mu_{2}$ where
$\mu_{1}$ is an atomic normalized POVM and $\mu_{2}$ is a non-atomic
normalized POVM and  they are mutually singular. Such a
decomposition is unique.
Furthermore $\mu_1$ and $\mu_2$ are
$C^*$-extreme and in particular  $\mu_1$ is spectral.
\end{theorem}
\begin{proof}
 Let $\{B_j\}_{j\in J}$ be a maximal collection of mutually disjoint atoms for $\mu$, which
 exists due to Zorn's lemma. As in the proof of Theorem \ref{atomic $C^*$-extreme points are PVM},
  since $\mu$ is $\cst$-extreme, we note using Lemma \ref{atoms are projections}  that   $\mu(B_j)$ is a projection for each $j$. Also $\{\mu(B_j)\}_{j\in J}$ are mutually orthogonal by Proposition \ref{prop:projection commutes with everything}. Since $\h$ is separable, it follows that $J$ is countable. This further implies that if we set $X_1=\cup _{j\in J}B_j$, then since
\begin{align}
    \mu(X_1)=\sum_{j\in J}\mu( B_j),
\end{align}
 $\mu(X_1)$ is a projection.
Now set $X_2=X\setminus X_1$. For $i=1,2$,  let $\h_i=\ran(\mu(X_i))$, and define the operator valued measures $\mu_i:\ox\to\B(\h_i)$  by
\begin{align*}
&\mu_i(A)=\mu(A\cap X_i)_{|_{\h_i}}=\mu(A)_{|_{\h_i}}~~\mbox{ for all }A\in\ox.
\end{align*}
 It is clear that each $\mu_i$ is a normalized POVM. Also $\h=\h_1\oplus\h_2$ and
\begin{align*}
    \mu=\mu_1\oplus\mu_2.
\end{align*}
Now we show that $\mu_1$ is atomic. Assume that $\mu_1(A)\neq0$ for some $A\in\ox$. Then $\mu(A\cap X_1)\neq 0$ and, since
\begin{align*}
    \mu(A\cap X_1)=\sum_{j\in J}\mu(A\cap B_j),
\end{align*}
 it follows that $\mu(A \cap B_j)\neq 0$ for some $j$ and hence $\mu_1(A\cap B_j)\neq0$. Therefore, as $B_j$ is an atom for $\mu$, $A\cap B_j$ is an atom for $\mu$. Consequently, as $\mu_1(A\cap B_j)\neq0$, it follows that $A\cap B_j$ is an atom for $\mu_1$. Thus we have  got an atom contained in the subset $A$  with $\mu_1(A)\neq0$, which shows that $\mu_1$ is atomic.
To prove that $\mu_2$ is non-atomic, let if possible, $A$ be an atom for $\mu_2$. Since $\mu_2$ is concentrated on $X_2$,  $A\cap X_2$ is an atom for $\mu_2$ and hence $A\cap X_2$ is an atom for $\mu$. But then  $\{B_j\}_{j\in J}\cup\{A\cap X_2\}$ is a collection of mutually disjoint atoms for $\mu$, violating the maximality of the collection $\{B_j\}_{j\in J}$. Thus we conclude that $\mu_2$ is non-atomic. It is clear that $\mu_1$ and $\mu_2$ are mutually singular.

To show the uniqueness, let $\nu_1\oplus\nu_2$ be another such
decomposition with atomic $\nu_1\in \p_{\K_1}(X)$ and non-atomic
$\nu_2\in \p_{\K_2}(X)$ where $\h=\K_1\oplus\K_2$. We shall show that $\K_i=\h_i$ and
$\nu_i=\mu_i$ for $i=1,2.$ Let $Y_1$ and $Y_2$ be disjoint measurable subsets
such that $\nu_i(A)=\nu_i(A\cap Y_i)$ for  all $A\in \ox$. We know from
Proposition \ref{atomic and non atomic povms are mutually singular}
that $\mu_1\perp\nu_2$ and $\mu_2\perp\nu_1$ and so $Y_1$ and $Y_2$
can be chosen so that $Y_1\cap X_2=Y_2\cap X_1=\emptyset$.
Therefore for each $i=1,2$, since both $\mu_i$ and $\nu_i$ are
concentrated on $X_i\cup Y_i$ and $(X_1\cup Y_1)\cap (X_2\cup Y_2)=\emptyset$, we can assume without loss of generality, that $X_i=Y_i$ (just replace $X_i, Y_i$ by $X_i\cup Y_i$). Further note that
\begin{align*}
   I_{\K_i} =\nu_i(Y_i)=\mu(Y_i)_{|_{\K_i}}=\mu(X_i)_{|_{\K_i}}={P_{\h_i}}_{|_{\K_i}},
\end{align*}
 where $P_{\h_i}$ denotes the projection of $\h$ onto $\h_i$. This implies  $\K_i\subseteq \h_i$.  By symmetry, we have $\h_i\subseteq \K_i$. Hence $\K_i=\h_i$. Similarly for all $A\in \ox$, we get
\begin{align*}
    \nu_i(A)=\nu_i(A\cap Y_i)=\mu(A\cap Y_i)_{|_{\K_i}}=\mu(A\cap X_i)_{|_{\h_i}}=\mu_i(A\cap X_i)=\mu_i(A)
\end{align*}
showing that $\nu_i=\mu_i$. The second statement  follows from
Theorem \ref{direct sum of C*extreme points with disjoint support is
$C^*$-extreme if and only if each component is $C^*$-extreme} and
Theorem \ref{atomic $C^*$-extreme points are PVM}.
\end{proof}

\begin{remark}
In the theorem above, we cannot expect a similar kind of direct sum
decomposition for a normalized POVM which is not  $\cst$-extreme. To
see an example, let $\lambda_1$ and $\lambda_2$ be two probability
measures on some measurable space $X$ such that $\lambda_1$ is
atomic while $\lambda_2$ is non-atomic. Let $T\in\bh$ be a positive
contraction which is not a projection. Consider the POVM $\mu\in\px$
defined by $\mu (\cdot )
=\lambda_1(\cdot)T+\lambda_2(\cdot)(I_\h-T)$. One can easily verify
that no decomposition of $\mu$ into a direct sum of atomic and
non-atomic normalized POVMs exists.
\end{remark}

One reason for us to study the notion of mutually singular POVMs  is
the following result. Its proof follows from Theorem \ref{thm:every povm
decomposes as direct sum of atomic and non atomic povm}  and
Theorem \ref{direct sum of C*extreme points with disjoint support is
$C^*$-extreme if and only if each component is $C^*$-extreme}.
 Since we have already  characterized all atomic $\cst$-extreme points (Theorem \ref{atomic $C^*$-extreme points are PVM}), it says
  in particular that it
   is sufficient to look for the characterization of  non-atomic $C^*$-extreme points to understand
   the general situation.

\begin{corollary}\label{iff criterion for C*-extreme points via atomic and non atomic povm}
Let $\mu:\ox\to\bh$ be a normalized POVM and let $X_1=\cup_{i\in I}B_i$ be the union of a maximal collection $\{B_i\}_{i\in I}$ of mutually disjoint atoms for $\mu$. Let $X_2=X\setminus X_1$.
Then $\mu$ is $\cst$-extreme in $\px$ if and only if
\begin{enumerate}
    \item the operators $\mu(X_1)$ and $\mu(X_{2})$ are projections and,
    \item   $\mu=\mu_1\oplus\mu_{2}$ such that $\mu_i$ is $C^*$-extreme in $\mathcal{P}_{\h_i}(X)$,
    where $\h_i=\ran(\mu(X_i))$ and $\mu_i=\mu(\cdot)_{|_{\h_i}}$ for $i=1,2$.
\end{enumerate}
\end{corollary}

\section{Measure Isomorphic  POVMs}\label{measure isomorphism}

We digress a bit from the earlier developments and explore
$\cst$-extreme properties from the perspective of measure
isomorphism.
 In classical measure theory, this notion has been examined extensively. The idea
  is to neglect  measure zero subsets in considering isomorphisms. 
One consequence is that most questions about abstract measure spaces
get reduced to questions about sub $\sigma $-algebras of the Borel
$\sigma$-algebra of the unit interval $[0,1].$ In a sense this space
is universal.

Measure isomorphism for POVMs seems to have been first studied in
\cite{Dorofeev Graaf}. Our
 aim here is quite limited  to investigate preservation
   of some natural properties of POVMs, especially $\cst$-extremity, under this isomorphism.
Here too we see the role of the unit interval.

Let $X$ be a measurable space and $\h$ a Hilbert space. Let $\mu:\ox\to\bh$ be a POVM. For
each $A\in\ox$, let $[A]_{\mu}$ denote the set
$$
    [A]_\mu:=\left\{B\in\ox; \mu(A\setminus B)=0=\mu(B\setminus A)\right\}
    =\{B\in\ox;\mu(B)=\mu(A)=\mu(B\cap A)\}. $$
Consider
\begin{align*}
    \Sigma(\mu):=\left\{[A]_\mu;A\in\ox\right\}.
\end{align*}
Then $\Sigma(\mu)$ is a Boolean $\sigma$-algebra under the following operations:
\begin{align}
   \label{eq:setminus relation} [A]_\mu\setminus [B]_\mu=[A\setminus B]_\mu\\
    \label{eq:intersection relation}[A]_{\mu}\cap[B]_\mu=[A\cap B]_\mu
\end{align}
for any $A,B\in\ox$.
Define $\tilde{\mu}:\Sigma(\mu)\to\bh$ by
\begin{align*}
    \tilde{\mu}([A]_\mu)=\mu(A)~~\mbox{for all }A\in\ox,
\end{align*}
which is well defined by virtue of the very
definition of $[A]_\mu$. If there is no possibility of confusion, we
shall still denote $\tilde{\mu}$ by $\mu$ only.

\begin{definition}(\cite{Dorofeev Graaf})
For $i=1,2$, let $X_i$ be two measurable spaces and let $\h$ be a Hilbert space.
Two POVMs $\mu_i:\mathcal{O}(X_i)\to\bh$ are called
\emph{measure isomorphic} and denoted by $\mu_1\cong \mu_2$, if
there exists a Boolean isomorphism
$\Phi:\Sigma(\mu_1)\to\Sigma(\mu_2)$ i.e. $\Phi$ is bijective and
both $\Phi$ and $\Phi^{-1}$ preserve  the  operations in \eqref{eq:setminus relation} and \eqref{eq:intersection relation}:
\begin{align}\label{eq:measure isomorphism eqations}
   \nonumber &\Phi\left([A_1]_{\mu_1}\setminus[B_1]_{\mu_1}\right)=\Phi([A_1]_{\mu_1})\setminus\Phi([B_1]_{\mu_1}),\\
    &\Phi\left([A_1]_{\mu_1}\cap[B_1]_{\mu_1}\right)=\Phi\left([A_1]_{\mu_1}\right)\cap\Phi\left([B_1]_{\mu_1}\right)\;etc.
\end{align}
such that $\mu_1\left(A_1\right)=\mu_2\left(\Phi([A_1]_{\mu_1})\right)$ for all $A_1,B_1\in\mathcal{O}(X_1)$.
\end{definition}

The following theorem compares some natural properties of POVMs under measure isomorphism.

\begin{theorem}\label{some similar properties of isomorphic POVM}
Let $\mu_i:\mathcal{O}(X_i)\to\bh$, $i=1,2$ be  two normalized POVMs such that they are measure isomorphic. Then we have the following:
\begin{enumerate}
\item  $\mu_1$ is a spectral measure if and only if  $\mu_2$ is a spectral measure.
\item  $\mu_1$ is atomic (non-atomic) if and only if $\mu_2$ is atomic (non-atomic).
\item  $\mu_1$ is $C^*$-extreme (extreme) in $\mathcal{P}_\h(X_1)$ if and only if $\mu_2$ is $C^*$-extreme (extreme) in $\mathcal{P}_\h(X_2)$.
\end{enumerate}
\end{theorem}
\begin{proof}
Let $\Phi:\Sigma(\mu_1)\to\Sigma(\mu_2)$ be a Boolean isomorphism satisfying $\mu_1(A_1)=\mu_2(\Phi([A]_{\mu_1}))$ for all $A_1\in\mathcal{O}(X_1)$. By symmetry, it is enough to prove the statements in just one direction.

(1) This is  straightforward by isomorphism. If $\mu_2$ is a spectral measure then for any $A_1\in \mathcal{O}(X_1)$, $\mu_2(\Phi([A_1]_{\mu_1}))$ is a projection. Since $\mu_1(A_1)=\mu_2\left(\Phi([A_1]_{\mu_1})\right)$, it follows that $\mu_1(A_1)$ is a projection and hence $\mu_1$ is a spectral measure.

(2)  Firstly we claim that if $A_1$ is an atom for $\mu_1$, then $A_2$ is an atom for $\mu_2$ for any $A_2\in \Phi([A_1]_{\mu_1})$. To see this, first note that $\mu_2(A_2)=\mu_1(A_1)\neq0$.
 Let $A_2'\subseteq A_2$ be a measurable  subset. Then for any  $ A'_1\in\Phi^{-1}([ A_2']_{\mu_2})$, we have
 \begin{align*}
    \Phi\left([A'_1\cap A_1]_{\mu_1}\right)= \Phi\left([A'_1]_{\mu_1}\right) \cap  \Phi\left([A_1]_{\mu_1}\right)= [A_2']_{\mu_2}\cap[A_2]_{\mu_2}=[A_2'\cap A_2]_{\mu_2}=[A_2']_{\mu_2}=\Phi([ A_1']_{\mu_1})
 \end{align*}
 and hence $[A'_1\cap A_1]_{\mu_1}=[A_1']_{\mu_1}$, which in turn implies
 \begin{align}\label{eq: measure of A_1' cap A_1}
 \mu_1(A_1'\cap A_1)=\mu_1(A_1').
 \end{align}
 But since $A_1$ is atomic for $\mu_1$, we have
 \begin{align*}
     \text{either }\mu_1(A_1'\cap A_1)=0\;\text{ or }\;\mu_1(A_1'\cap A_1)=\mu_1(A_1)
 \end{align*}
 and therefore from \eqref{eq: measure of A_1' cap A_1},
  \begin{align*}
     \text{either }\mu_1(A_1')=0\;\text{ or }\;\mu_1(A_1')=\mu_1(A_1).
 \end{align*}
 Since $A_1\in \Phi^{-1}([A_2]_{\mu_2})$ and $A'_1\in \Phi^{-1}([A_2']_{\mu_2})$, it follows that
  \begin{align*}
     \text{either }\mu_2(A_2')=0\;\text{ or }\;\mu_2(A_2')=\mu_2(A_2).
 \end{align*}
 This shows our claim that $A_2$ is an atom for $\mu_2$.

Now assume that $\mu_1$ is atomic. To show that $\mu_2$ is atomic, let $A_2\in\mathcal{O}(X_2)$ be such that $\mu_2(A_2)\neq 0$. If $A_1\in \Phi^{-1}([A_2]_{\mu_2})$, then $\mu_1(A_1)=\mu_2(A_2)\neq0$. Since $\mu_1$ is atomic, $A_1$ contains an atom for $\mu_1$, say $A_1'$. Fix $A_2'\in \Phi([A_1']_{\mu_1})$. Then $A_2'$ is an atom for $\mu_2$ by the claim above. As above we show that $\mu_2(A_2'\cap A_2)=\mu_2(A_2')$, which implies that $A_2'\cap A_2$ is an atom for $\mu_2$ contained in $A_2$. This proves that $\mu_2$ is atomic. Similarly if $\mu_1$ is non-atomic, then there is no atom for $\mu_1$, and again it follows from the claim above that there is no atom for $\mu_2$, which is equivalent to saying that $\mu_2$ is non-atomic.

(3) Assume that $\mu_2$ is $C^*$-extreme in $\mathcal{P}_{\h}(X_2)$. To show that $\mu_1$ is $\cst$-extreme in $\p_\h(X_1)$, let $\mu_1(\cdot)=\sum_{j=1}^nT_j^*\mu_1^j(\cdot)T_j$ be a proper  $C^*$-convex combination in $\mathcal{P}_\h(X_1)$. For each $j$, define $\mu_2^j:\mathcal{O}(X_2)\to\bh$ by
\begin{align*}
    \mu_2^j(A_2)=\mu_1^j\left(\Phi^{-1}\left([A_2]_{\mu_2}\right)\right)~~\text{for all }A_2\in\mathcal{O}(X_2).
\end{align*}
For $\mu_2^j$ to be  well defined, we need to show that $\mu_1^j(A_1)=\mu_1^j( A_1')$ for any  $A_1,  A_1'\in \Phi^{-1}([A_2]_{\mu_2})$. So fix $A_1,  A_1'\in \Phi^{-1}([A_2]_{\mu_2})$. Then   $[A_1]_{\mu_1}=[ A_1']_{\mu_1}$ and hence, we get
 \begin{align*}
    \mu_1(A_1\setminus A_1')=0=\mu_1( A_1'\setminus A_1).
\end{align*}
 Therefore, since  $T_j^*\mu_1^j(\cdot)T_j\leq\mu_1(\cdot)$, it follows that
 \begin{align*}
     T_j^*\mu_1^j(A_1\setminus A_1')T_j=0=T_j^*\mu_1^j( A_1'\setminus A_1)T_j
 \end{align*}
 which, as $T_j$ is invertible, yields
 \begin{align*}
     \mu_1^j(A_1\setminus A_1')=0=\mu_1^j( A_1'\setminus A_1).
 \end{align*}
This implies the requirement for well-definedness of $\mu_2^j$. Also note that
\begin{equation}
\mu_1^j(A_1)=\mu_1^j\left(\Phi^{-1}(\Phi([A_1]_{\mu_1})\right)=\mu_2^j(\Phi([A_1]_{\mu_1})),\label{eq:mu_1 representation}
\end{equation}
for all $A_1\in\mathcal{O}(X_1)$.
Further for any $A_2\in\mathcal{O}(X_2)$, we have
\begin{align*}
     \sum_{j=1}^nT_j^*\mu_2^j(A_2)T_j=\sum_{j=1}^nT_j^*\mu_1^j\left(\Phi^{-1}([A_2]_{\mu_2})\right)T_j=\mu_1\left(\Phi^{-1}([A_2]_{\mu_2})\right)=\mu_2(A_2).
\end{align*}
Subsequently, since $\mu_2$ is $C^*$-extreme in $\mathcal{P}_\h(X_2)$, there exists an unitary operator $U_j\in\bh$ such that $\mu_2(\cdot)=U_j^*\mu_2^j(\cdot)U_j$ for each $j$. It then follows for all $A_1\in\mathcal{O}(X_1)$, that
\begin{align*}
    \mu_1(A_1)=\mu_2(\Phi([A_1]_{\mu_1}))=U_j^*\mu_2^j(\Phi([A_1]_{\mu_1}))U_j=U_j^*\mu_1^j(A_1)U_j,
\end{align*}
 where the last equality is due to   \eqref{eq:mu_1 representation}. This proves that $\mu_1$ is unitarily equivalent to each $\mu_1^j$ which consequently implies that $\mu_1$ is $C^*$-extreme in $\mathcal{P}_\h(X_1)$. That $\mu_1$ is extreme if and only if $\mu_2$ is extreme follows similarly.
\end{proof}

In the proof of part (2) of the theorem above, we observed the following:
\begin{proposition}
Let $\mu_i:\mathcal{O}(X_i)\to\bh$, $i=1,2$ be two measure isomorphic POVMs with Boolean isomorphism $\Phi:\Sigma(\mu_1)\to\Sigma(\mu_2)$. Then  $A_1$ is an atom for $\mu_1$ if and only if any representative of $\Phi([A_1]_{\mu_1})$ is an atom for $\mu_2$.
\end{proposition}

 Let $\mu:\ox\to\bh$ be a POVM. We say $\mu$ is {\em countably generated} if there exists a countable collection of subsets $\mathcal{F}\subseteq\ox$ such that for any $A\in\ox$, there exists $B\in\sigma(\mathcal{F})$ satisfying $[A]_\mu=[B]_\mu$. Here $\sigma(\mathcal{F})$ denotes the $\sigma$-algebra generated by $\mathcal{F}$. The following result has been borrowed from \cite{Beukema}.

\begin{theorem}\label{a povm is isomorphic to a povm on Borel sigma algebra of [0,1]}(Proposition 59, \cite{Beukema})
If $\mu:\ox\to\bh$ is a countably generated POVM, then $\mu$ is measure isomorphic to a POVM $\nu:\mathcal{O}([0,1])\to\bh$.
\end{theorem}

Recall that when $X$
is a separable metric space, then $\ox$ is its Borel $\sigma$-algebra and in this case, any POVM on $X$ is countably generated.
What the theorem above basically says is that, to study $\cst$-extreme points in $\px$ for a separable metric space $X$, it is sufficient to just characterize the $\cst$-extreme points in $\p_\h([0,1])$ in view of Theorem \ref{some similar properties of isomorphic POVM}. This result will also help us find an example (see Example \ref{existence of a non-homomorphic C^*-extreme points on an uncountable metric space}) of a $\cst$-extreme point in $\px$ which is not spectral, when $\h$ is infinite dimensional.

Now we consider  measure isomorphism of POVMs induced from a
bimeasurable map. Recall that for measurable spaces $X_1$ and $X_2$, a function $f:X_1\to X_2$ is called \emph{measurable} if $f^{-1}(A_2)\in\mathcal{O}(X_1)$ whenever $A_2\in\mathcal{O}(X_2)$.
Note that for any measurable space $X $ and a measurable subset $Y\subseteq X$, $Y$
 itself inherits the natural measurable space structure from $X$ with  the $\sigma$ algebra $\{A\cap Y; A\in\ox\}$.

\begin{theorem}\label{Borel isomorphism preserve same properties}
For  $i=1,2$, let $X_i$ be two measurable spaces and let
$Y_i\subseteq X_i$ be measurable subsets. Let $f:Y_1\to Y_2$ be a
bijective map such that both $f$ and $f^{-1}$ are measurable. Given
a normalized POVM $\mu_1:\mathcal{O}(X_1)\to\bh$ satisfying
$\mu_1(A_1)=\mu_1\left(A_1\cap Y_1\right)$ for all
$A_1\in\mathcal{O}(X_1)$, define $\mu_2:\mathcal{O}(X_2)\to\bh$ by
$\mu_2(A_2)=\mu_1\left(f^{-1}(A_2\cap Y_2)\right)$ for all
$A_2\in\mathcal{O}(X_2)$. Then $\mu _1$ and $\mu _2$ are measure
isomorphic.
\end{theorem}
\begin{proof}
We claim that the map $\Phi:\Sigma(\mu_1)\to\Sigma(\mu_2)$ defined
by
\begin{align}\label{eq:definition of Phi}
    \Phi([A_1]_{\mu_1})=[f(A_1\cap Y_1)]_{\mu_2}~~\text{for all }A_1\in\mathcal{O}(X_1),
\end{align}
is a Boolean isomorphism. First note that
\begin{align}\label{eq: mu_1 expressed in terme of mu_2}
\mu_1(A_1)=\mu_1(A_1\cap Y_1)=\mu_1\left(f^{-1}\left(f(A_1\cap Y_1)\right)\right)=\mu_2\left(f(A_1\cap Y_1)\right)
\end{align}
for all $A_1\in\mathcal{O}(X_1)$. This implies that $\mu_1(A_1)=0$ if and only if $\mu_2(f(A_1\cap Y_1))=0$ for any $A_1\in\mathcal{O}(X_1)$. Therefore if $[A_1]_{\mu_1}=[A_1']_{\mu_1}$ for some $A_1,A_1'\in\mathcal{O}(X_1)$, then $[f(A_1\cap Y_1)]_{\mu_2}=[f(A_1'\cap Y_1)]_{\mu_2}$. This proves the well-definedness of $\Phi$. Similarly by symmetry, we prove that $\Phi$ is injective. That $\Phi$ is onto is straightforward by noting that
$$\Phi\left([f^{-1}(A_2\cap Y_2)]_{\mu_1}\right)=[A_2\cap Y_2]_{\mu_2}=[A_2]_{\mu_2}$$
for any $A_2\in\mathcal{O}(X_2)$.   This shows that $\Phi$ is a Boolean isomorphism as claimed. Further from \eqref{eq:definition of Phi} and \eqref{eq: mu_1 expressed in terme of mu_2}, we have
$$\mu_2(\Phi([A_1]_{\mu_1}))=\mu_2(f(A_1\cap Y_1))=\mu_1(A_1)$$
for any $A_1\in\mathcal{O}(X_1)$.
Thus we conclude that $\mu_1$ and $\mu_2$ are measure isomorphic.
\end{proof}

Now we apply these results to the study of $C^*$-extreme POVMs.
Consider the  map $g:[0,1)\to\T$ given by
$g(t)=e^{2\pi it}$ for $t\in[0,1)$, where $\T$ is the unit circle.
It is clear that $g$ is a bijective map  such that both $g$ and $g^{-1}$ are Borel measurable.
Therefore for any Hilbert space $\h$, normalized POVMs $\mu\in \p_\h([0,1])$ with
$\mu(\{1\})=0$ are in one-to-one correspondence with $\p_\h(\T)$ through measure isomorphism, by  Theorem  \ref{Borel isomorphism preserve
same properties}. In particular, since singletons under non-atomic POVMs have zero measure, it follows that non-atomic POVMs in $\p_\h([0,1])$ are measure isomorphic to non-atomic POVMs in $\p_\h(\T)$.

Next if $X$ is a separable metric space,
then non-atomic POVMs in $\px$ are measure isomorphic to  non-atomic
POVMs in $\p_\h([0,1])$ from Theorem \ref{a povm is isomorphic to a povm on Borel sigma algebra of [0,1]} and Theorem \ref{some similar properties of isomorphic POVM}, which in turn are measure isomorphic to
non-atomic POVMs in $\p_\h(\T)$ as seen above. Thus we conclude in  view of Theorem \ref{some similar properties of isomorphic POVM} that, characterizing the non-atomic
$\cst$-extreme points in $\px$ is equivalent to characterizing
non-atomic $\cst$-extreme points in $\p_\h([0,1])$ or $\p_\h(\T)$. Also we already know the structure of atomic $\cst$-extreme points. Therefore what we observed from the discussion above and Corollary \ref{iff criterion for C*-extreme points via atomic and non atomic povm} is that, to characterize  $\cst$-extreme points of $\px$, it is enough to understand the behaviour of $\cst$-extreme points of $\p_\h([0,1])$ or $\p_\h(\T)$.

\section{POVMs on Topological Spaces}\label{povm on topological space}

The results presented  in this article so far have been for POVMs on general
measurable spaces. Our attention now shifts toward the particular case
of topological spaces. For the whole section, we assume that $X$ is a Hausdorff topological space. As mentioned earlier, in this case $\ox$  will  denote the Borel $\sigma$-algebra of $X$.

\subsection{Regular POVMs}
An additional property of a POVM that can be studied when $X$ is a
topological space, is that of regularity. The assumption of
regularity shall be useful once we discuss the correspondence
between POVMs and completely positive maps in Section 7. Recall that a positive measure $\lambda$ is regular if  it is inner regular (or tight) with respect to  compact
subsets and outer regular  with respect to open subsets:
\begin{align*}
    \lambda(A)&=\sup\{\lambda(E): E\text{  compact with } E\subseteq A\}\\
    &=\inf\{\lambda(G): G\text{ open with
    }A\subseteq G\},
\end{align*}
for every $A\in\ox$.

\begin{definition}\label{definition of regularity}
A POVM $\mu:\ox\to\bh$ on a topological space $X$ is said to be {\em regular}
if $\mu_{h,h}$ as defined in equation
\eqref{eq:notation for mu_h,k}, is a regular positive measure  for each $h\in \h$.
\end{definition}

The issue of regularity does not arise
for complete separable metric spaces (Theorem 3.2, \cite{KRP3}), as all Borel measures are automatically regular.

The following lemma says that regularity is
preserved under the minimal Naimark dilation.

\begin{lemma}\label{mu is regular iff pi is regular}
Let  $\mu:\ox\to\bh$ be a POVM with the minimal Naimark dilation
$(\pi, V,\hpi)$. Then $\mu$ is regular if and only if $\pi$ is
regular.
\end{lemma}
\begin{proof}
If $\pi$ is regular then, since $\mu_{h,h}=\pi_{Vh,Vh}$ for each
$h\in\h$, it is clear that $\mu$ is regular. For the converse,
assume that $\mu$ is regular. First note that, if $k=\pi(B)Vh$ for
some $B\in\ox,h\in \h$, then for any $A\in \ox ,$
\begin{align*}
    \pi_{k,k}(A)=\langle \pi(B)Vh,\pi(A)\pi(B)Vh\rangle=\langle h,V^*\pi(A)\pi(B)Vh\rangle=\mu_{h,h}(A\cap
    B).
\end{align*}
Since $A\mapsto \mu_{h,h}(A\cap B)$ is regular, it follows that
$\pi_{\pi(B)Vh,\pi(B)Vh}$ is regular. Consequently, $\pi_{k,k}$ is
regular for all $k\in \Span\{\pi(A)Vh: A\in\ox,h\in\h\}$. Now fix $\epsilon>0$ and $B\in\ox$. Then for
general $k\in\hpi$, let $\{k_0\}$ be  in $\Span\{\pi(A)Vh : A\in\ox,h\in\h\}$
such that
$$\|k-k_0\|<\sqrt{\epsilon}/2.$$
Since $\pi_{k_0,k_0}$ is
regular as shown above, there is a compact subset $C$ and an open subset $O$ with
$C\subseteq B\subseteq O$ such that $$\langle k_0,\pi(O\setminus
C)k_0\rangle<\epsilon/4.$$
Thus
\begin{align*}
    \langle k,\pi(O\setminus C)k\rangle&=\|\pi(O\setminus C)^{1/2}k\|^2
    \leq 2\|\pi(O\setminus C)^{1/2}k_0\|^2+2\|\pi(O\setminus C)^{1/2}(k_0-k)\|^2\\
    &\leq2\langle k_0,\pi(O\setminus C)k_0\rangle+2\|k_0-k\|^2
<2\left(\epsilon/4+\epsilon/4\right)=\epsilon.
\end{align*}
Since $\epsilon$ and $B$ are arbitrary, we conclude that $\pi_{k,k}$
is regular.
\end{proof}

\begin{remark}\label{regularity of T muT and nu leq mu}
If $\mu$ is a regular POVM, then it is  easy to check that
$T^*\mu(\cdot)T$ is also regular for any $T\in \bh$. Moreover, if
$\nu $ is a POVM such that $\nu\leq\mu ,$ then $\nu$ is also
regular.
\end{remark}

\subsection{Regular atomic and non-atomic POVMs}
We now discuss the structure of atomic and non-atomic regular POVMs.
Just like that in classical theory, we show that every atom for a
regular POVM is concentrated on a singleton up to a set of measure
$0$ and that every atomic regular POVM  is concentrated on a
countable subset. First step in that direction is the following
lemma.

\begin{lemma}\label{a PVM is dirac if it is scalar valued}
Let $\pi:\ox\to\B(\hpi)$ be a regular spectral measure satisfying
$\pi(A)=I_{\hpi}$ or $0$ for each $A\in\ox$ (here $\hpi $ could be non-separable). Then there exists a
unique $x\in X$ such that $\pi=\delta_{x}(\cdot)I_{\hpi}$, where $\delta_x$
denotes the Dirac measure concentrated at $x$.
\end{lemma}
 \begin{proof}
 For each $A\in \ox$, let $\lambda(A)=0$ or $1$ accordingly so that $\pi(A)=\lambda(A)I_{\hpi}$. Clearly $\lambda$ is a
 regular probability measure, as $\pi$ is regular (e.g. $\lambda=\pi_{h,h}$ for any unit vector $h\in \hpi$).
 Whence by inner regularity, there is a compact subset $C\subseteq X$ such that $\lambda(C)>0$ and thus, $\lambda(C)=1$.
 We claim to find an element $x\in C$ such that $\lambda=\delta_x$. Suppose this is not
  the case, then $\lambda(\{x\})=0$ for each $x\in C$ (otherwise, $\lambda(\{x\})=1=\lambda(C)$ for
   some $x$). Therefore it follows from outer regularity of $\lambda$, that  there is an open subset $E_x $ containing $x$ such that $\lambda(E_x)<1/2$ and thus, $\lambda(E_x)=0$. Since $\{E_x\}_{x\in C}$ is an open cover for the compact subset $C$, there exist finitely many points $x_1,\ldots,x_n\in C$ such that $C\subseteq\cup_{i=1}^nE_{x_i}$.
   But then we have $$\lambda(C)\leq\sum_{i=1}^n\lambda(E_{x_i})=0,$$
 leading us to a contradiction. Thus $\lambda=\delta_x$ for some $x\in X$ and hence $\pi=\delta_x(\cdot)I_{\hpi}$. The uniqueness is obvious as $\lambda
 (X)=\lambda(\{x\})=1.$
 \end{proof}

It is well-known  that the lemma above fails to be true  (even on
compact Hausdorff spaces) for finite positive measures, if we drop the
regularity assumption  (see Example 7.1.3, \cite{Bogachev}).


The following theorem and the subsequent corollary give characterization of all  atomic and non-atomic regular
POVMs.

 \begin{theorem}\label{how atomic POVM looks like}
 Let $\mu:\ox\to\bh$ be an atomic regular POVM. Then there exists a countable subset $\{x_n\}$ of $X$ and positive operators $\{T_n\}$ in $\bh$ such that
 \begin{align}
     \mu(A)=\sum_n\delta_{x_n}(A)T_n
 \end{align}
for each $A\in\ox$.
\end{theorem}
 \begin{proof}
 Let $(\pi,V,\hpi)$ be the minimal Naimark dilation for $\mu$. Since $\mu$ is atomic,
   $\pi$ is also atomic  by Proposition \ref{mu is atomic iff pi is atomic}. Also $\pi$ is regular
   by Lemma \ref{mu is regular iff pi is regular}. We claim that there exists a countable
   subset $\{x_n\}$ of $X$ and orthogonal projections $\{P_n\}$ on $\hpi$ such that $\pi(A)=\sum_n\delta_{x_n}(A)P_n$
   for all $A\in\ox$. Then the required assertion will follow by taking $T_n=V^*P_nV\in\bh$.

Let $\{B_i\}_{i\in I}$ be a maximal collection of mutually disjoint  atoms for $\pi$,
whose  existence is ensured  by  Zorn's lemma. Note that, since $\pi(B_i)\neq0$, we have $\mu(B_i)\neq 0$ for all $i\in I$ by Proposition \ref{zero sets of mu and pi are same}. Hence it follows from Remark \ref{a collection of disjoint sets with non zero measure must be countable}  that $I$ must be countable.  Furthermore  for each $A\in\ox$, we have
    \begin{align}\label{eq:pi is expressed using B_alpha}
        \pi(A)=\sum_{i\in I} \pi(A\cap B_i),
    \end{align}
otherwise there would exist an atom for $\pi$, say $ A_1\subseteq A\setminus (\cup_i(A\cap B_i))$ which is disjoint to each $B_i$, violating the maximality of the collection $\{B_i\}_{i\in I}$.

Now for each $i\in I$, we set $P_i=\pi(B_i)$ and $\h_i=\ran(P_i)$. Note that  each $\h_i$ is a reducing subspace for $\pi(A)$  for all $A\in\ox$
and therefore, the map $\pi_i:\ox\to\mathcal{B}(\h_i)$ given by
\begin{align*}
          \pi_i(A)=\pi(A\cap B_i)_{|_{\h_i}}=\pi(A)_{|_{\h_i}}~~\text{for all }A\in\ox,
\end{align*}
is a well defined regular spectral measure.  Also for  each $A\in\ox$, as $B_i$ is an atom for $\pi$, we have
$$\text{either }\pi(A\cap B_i)=0\text{ or } \pi(A\cap B_i)=\pi(B_i)$$
that is
\begin{align*}
      \text{either }  \pi_i(A)=0 \text { or }\pi_i(A)=I_{\h_i}.
\end{align*}
Therefore by Lemma \ref{a PVM is dirac if it is scalar valued}
there is an element   $x_i\in B_i$ such that
$\pi_i=\delta_{x_i}(\cdot)I_{\h_i}$. Equivalently for each $A\in\ox$, we have
\begin{align*}
    \pi(A)P_i=\delta_{x_i}(A)P_i
\end{align*}
and hence
 \eqref{eq:pi is expressed using B_alpha} yields
\begin{align*}
        \pi( A)=\sum_{i\in I}\pi(A\cap B_i)=\sum_{i\in I}\pi(A)\pi(B_i)=\sum_{i\in I}\delta_{x_i}(A)P_i.
\end{align*}
This shows our claim, completing the proof.
 \end{proof}

 \begin{corollary}\label{if and only if criteria for atomic and non-atomic measures}
Let $\mu:\ox\to\bh$ be a regular POVM. Then
\begin{enumerate}
    \item for any atom $B$ for $\mu$, there exists a (unique) $x\in B$ such that $\mu(B)=\mu(\{x\})$.
    \item \label{condition for atomicity}  $\mu$ is atomic if and only if there exists a countable subset $Y\subseteq X$ such that $\mu(Y)=\mu(X)$.
 \item\label{condition for non-atomicity}  $\mu$ is non-atomic if and only if $\mu(\{x\})=0$ for all $x\in X$.
\end{enumerate}
\end{corollary}
\begin{proof}
The proof of (1) is actually ingrained in the proof of Theorem \ref{how atomic POVM looks like}; if $B$ is an atomic subset for $\mu$, then it is atomic for $\pi$ and then we actually showed above that $\pi(B)=\pi(\{x\})$ for some $x\in B$. To prove \eqref{condition for atomicity}, first note that any POVM concentrated on a countable subset is atomic and hence the `if' part follows. The converse follows from  Theorem \ref{how atomic POVM looks like}, by taking $Y=\{x_n\}$.
The  `only if' of Part \eqref{condition for non-atomicity} is trivial. To prove  the `if' part of (3), since every atom is concentrated on a singleton by Part (1),
the hypothesis implies that $\mu$ has no atom, which is equivalent to saying that $\mu$ is non-atomic.
\end{proof}

\begin{corollary}\label{Cor:atomic/nonatomic are preservaed in direwct sum}
Let $\{\mu_n\}$ be a countable collection of regular POVMs and let $\mu=\oplus_n\mu_n$. Then $\mu$ is atomic (non-atomic) if and only if each $\mu_n$ is atomic (non-atomic).
\end{corollary}
\begin{proof}
We use Corollary \ref{if and only if criteria for atomic and non-atomic measures} to prove the assertions. If $\mu$ is atomic, then there is a countable subset $Y$ such that $\mu(Y)=\mu(X)$. In particular $\mu_n(Y)=\mu_n(X)$ for each $n$, which implies that $\mu_n$ is atomic. Conversely if each $\mu_n$ is atomic, then $\mu_n(Y_n)=\mu_n(X)$ for some countable subset $Y_n$. If $Y=\cup_n Y_n$, then $Y$ is countable and $\mu(Y)=\mu(X)$, concluding that $\mu$ is atomic. The equivalence of non-atomicity follows similarly.
\end{proof}

\subsection{Regular $\cst$-extreme POVMs} For any topological space $X$ and a Hilbert space $\h$, denote the collection of all regular normalized POVMs from $\ox$ to $\bh$ by $\rpx$. Note  that $\rpx\subseteq\px$ and $\rpx$ is itself a $\cst$-convex set in the sense that $$\sum_{i=1}^nT_i^*\mu_i(\cdot)T_i\in \rpx,$$
whenever $\mu_i\in\rpx$ and $T_i$'s are $\cst$-coefficients for
$1\leq i\leq n$. In a fashion similar to Definition \ref{definition
of C*-extreme points}, we can define  $\cst$-extreme points of
$\rpx$. The following proposition says that, for a regular
normalized POVM $\mu$, it does not matter whether we are considering
$\cst$-extremity of $\mu$ in $\rpx$ or in $\px$.

\begin{proposition}\label{a regular povm is C*-extreme in px iff it is in rpx}
Let $\mu:\ox\to\bh$ be a normalized regular POVM. Then $\mu$ is $\cst$-extreme (extreme) in $\px$ if and only if $\mu$ is $\cst$-extreme (extreme) in $\rpx$.
\end{proposition}
\begin{proof}
 If we show that  every proper $\cst$-convex combination for $\mu$ in  $\px$ is also a proper $\cst$-convex combination in $\rpx$ and vice versa, then we are done.
So let  $\mu(\cdot)=\sum_{i=1}^nT_i^*\mu_i(\cdot)T_i$ be a proper $\cst$-convex combination in $\px$ for $\mu_i\in \px$.
Note that, since $T_i^*\mu_i(\cdot)T_i\leq \mu(\cdot)$ for each $i$, it follows from Remark \ref{regularity of T muT and nu leq mu} that $T_i^*\mu_i(\cdot)T_i$ is regular. Again by the same remark, since
\begin{align*}
    \mu_i(\cdot)=T^{*^{-1}}\left(T^*\mu_i(\cdot)T_i\right)T_i^{-1},
\end{align*}
it follows that $\mu_i$ is regular. Thus $\mu_i\in\rpx$, which shows that $\sum_{i=1}^nT_i^*\mu_i(\cdot)T_i$ is also a proper $\cst$-convex combination for $\mu$ in $\rpx$. Since $\rpx\subseteq\px$, the converse of the claim is immediate. The assertions about extreme points follow similarly.
\end{proof}

We have already seen the following result for countable measurable
spaces in Theorem \ref{atomic $C^*$-extreme points are PVM} without
the assumption of regularity. The extension to uncountable discrete spaces
requires regularity in a crucial way.

\begin{proposition}\label{regular C*-extreme on discrete spaces are PVM}
Let $X$ be a discrete (possibly uncountable) space. Then every regular POVM on $X$ is atomic.  Moreover, a normalized POVM in $\rpx$ is $\cst$-extreme if and
only if it is spectral.
\end{proposition}
\begin{proof}
Firstly let $\lambda$ be a regular Borel positive measure on
$X.$ By regularity of $\lambda$, for each $n\in \N$ there is a compact subset $C_n$ such
that $\lambda(X\setminus C_n)<1/n$. Set $C=\cup_nC_n$. Since
$X$ is discrete, each of $C_n$ is a finite subset and hence $C$ is
countable. Note that
\begin{align*}
    \lambda(X\setminus C)\leq \lambda(X\setminus C_n)\leq 1/n,
\end{align*}
for each $n$ and hence, $\lambda(X\setminus C)=0$. This says that
every  regular Borel  positive measure on $X$ is concentrated
on a countable subset and so it is atomic.

Now let $\mu:\ox\to\bh$ be a regular POVM. Then as observed above,
 $\mu_{h,h}$ is concentrated on a countable subset for each $h\in \h$. Let
$\{h_n\}$ be an orthonormal basis for $\h.$ Then for each $n\in \N$, there are countable
subsets, say $B_n$ such that
$$\mu_{h_n,h_n}(X\setminus B_n)=0.$$
 Set $B=\cup_nB_n$, then for each $n\in \N$, we have
$$\mu_{h_n,h_n}(X\setminus B)\leq\mu_{h_n,h_n}(X\setminus B_n)=0.$$
Consequently for all $h\in\h$, we have
$$\mu_{h,h}(X\setminus B)=\langle h,\mu(X\setminus B)h\rangle=\sum_{n}|\langle h_n,h\rangle|^2\langle h_n,\mu(X\setminus B)h_n\rangle=0$$
and hence $\mu(X\setminus B)=0$. Since $B$ is countable, we conclude that $\mu$ is concentrated on a countable subset and so $\mu$ is atomic.  Thus if $\mu$ is a $\cst$-extreme point in $\rpx$, then it is  spectral by Theorem \ref{atomic $C^*$-extreme points are PVM}.
\end{proof}

\begin{remark}\label{C*-extreme point on one-point compactification of discrete space}
Let $\tilde{X}=X\cup\{\infty\}$ be the one-point compactification of
a discrete space $X$ and let $\mu:\mathcal{O}(\tilde{X})\to\bh$ be a
normalized regular POVM. Then the restriction $\mu_{|_{\ox}}$ of
$\mu$ to $\ox$ is also regular and hence concentrated on a countable
subset, as seen in  Proposition \ref{regular C*-extreme on discrete spaces are PVM}. In particular, $\mu$ itself
is concentrated on a countable subset and hence is atomic. Therefore,
we conclude from Theorem \ref{atomic $C^*$-extreme points are PVM} that  any regular normalized POVM on $\tilde{X}$ is
$\cst$-extreme  in $\mathcal{P}_\h(\tilde{X})$ if and only if
it is  spectral.
\end{remark}

\subsection{Topology on $\p_\h(X)$}
As earlier $X$ is a topological space and  $POVM_\h(X)$ denotes the
collection  of all POVMs from $\ox$ to $\bh$. Now we define a
topology on this set. We shall call this topology as `bounded-weak'
inspired from a topology defined on the collection of all completely
positive maps on a $\cst$-algebra with the same name.  The reason
for  this nomenclature  will be apparent in the next section.  We
have observed that the set $\px $ of normalized POVMs is a
$C^*$-convex set.  Our aim now is to show a Krein-Milman type
theorem for $\cst$-convexity in this topology on $\px$.

Let $\cbx$ denote the space of all bounded continuous functions on
$X$.
  Recall that $\mu_{h,k}$ is the complex measure as in \eqref{eq:notation for mu_h,k} for any POVM $\mu$.
  We define the topology by defining convergence of nets.

\begin{definition}\label{definition of bw topology on povm}
Given a net $\mu^i$ and $\mu$ in $POVM_\h(X)$, we say  $\mu^i\to\mu$
in $POVM_\h(X)$ in {\em bounded weak topology\/} if
\begin{align*}
    \int_Xfd\mu^i_{h,k}\to\int_X fd\mu_{h,k}
\end{align*}
 for all $f\in C_b(X)$ and $h,k\in\h$.
\end{definition}

Notice  that the topology on $POVM_\h(X)$ is the smallest
topology
 which makes the maps: $\mu\mapsto \int_X fd\mu_{h,k}$ from $POVM_\h(X)$ to $\C$, continuous for all $f\in C_b(X)$ and $h,k\in\h$. It is then  immediate to verify that, for a given $\mu\in POVM_\h(X)$, sets of the form
\begin{align}
    O=\left\{\nu\in POVM_\h(X); \left|\int_X f_id\nu_{h_i,k_i}-\int_X f_id\mu_{h_i,k_i}\right|<\epsilon, 1\leq i\leq n\right\},
\end{align}
where $f_i\in C_b(X)$, $h_i,k_i\in\h$ for $1\leq i\leq n$, $\epsilon
>0,$ form a basis around $\mu$ in $POVM_\h(X)$. The definition here reminds us the weak topology considered in classical probability
theory. Moreover, we shall see in Section \ref{application to completely positive maps} that this definition is directly connected to  the bounded
weak topology on  the collection of completely positive maps on a
commutative $\cst$-algebra.

It should be added here that one can define a topology on
$POVM_\h(X)$ in several ways.  For example,  for a net $\mu_i$ of
POVMs and a POVM $\mu$,
  we could define the convergence  $\mu_i\to\mu$ by saying
  that $\mu_i(A)\to\mu(A)$ in WOT (or $\sigma$-weak topology) for all $A\in \ox$. This topology is certainly stronger
  than the bounded weak topology defined above. This topology has been considered in \cite{Holevo}.
We could  have  also defined a topology just by considering
$C_c(X)$, the space of all compactly supported continuous functions,
instead of $\cbx$ in the definition. In this case, we would get a
weaker topology than we originally defined. Nevertheless in this
case, one can show along the lines of classical probability theory
that this topology agrees with bounded weak topology
 on $\px$ whenever $X$ is a locally compact Hausdorff space.

Our main focus for this topology is the set of normalized POVMs. In general,
the set $\px$ is not Hausdorff; for an example, one can consider the classically famous Dieudonn\'{e} measure $\lambda$ (which is  not regular) on the compact Hausdorff space $X=[0, \omega_1]$ equipped with order topology, where $\omega_1$ is the first uncountable ordinal (see Example 7.1.3, \cite{Bogachev}). One can show that $\int_X fd\lambda=f(\omega_1)=\int_Xfd\delta_{\omega_1}$ for all $f\in C_b(X)$ and hence the distinct elements $\lambda(\cdot) I_\h$ and $\delta_{\omega_1}(\cdot)I_\h$ in $\px$ are not separated by open subsets. However the topology restricted to $\rpx$ is Hausdorff whenever $X$ is a locally compact Hausdorff space, which is a consequence of uniqueness of regular Borel measures in Riesz-Markov theorem.

 \begin{remark}\label{X is compact iff rpx is compact}
 As in classical probability theory, for a locally compact Hausdorff space
 (more generally for completely regular space, see Lemma 8.9.2, \cite{Bogachev}), the set $\{\delta_x(\cdot)I_\h; x\in X\}$ is
 closed in $\rpx$ and   it is homeomorphic to $X.$ Using this or otherwise,
 one can show that $\rpx$ is compact if and only if $X$ is compact.
 \end{remark}

\subsection{A Krein-Milman type theorem}
 Now we move on to prove the main result of this section. It is well  known that, in a locally convex topological vector space,
 a convex compact set is the  closure of convex hull of its extreme points. This is known as Krein-Milman  theorem.
  We here establish a similar kind of result for $\cst$-convexity in the sense that $\px$ is the closure of $\cst$-convex hull of its $\cst$-extreme points. A Krein-Milman type theorem was proved in
   \cite{Farenick Plosker Smith} when $X$ is a compact Hausdorff  space and $\h$ is a finite dimensional Hilbert space. We generalize it to arbitrary topological spaces and arbitrary Hilbert spaces. Moreover, in our case
   the compactness of $\px$ is not required. We first consider the following proposition, whose proof
   follows the same argument as normally used in classical measure theory.
   We provide the proof for the
   sake of completeness.

\begin{proposition}\label{prop:normalized POVM with finite support are dense}
Let $X$ be a topological space and $\h$ a Hilbert space. Then the collection of all normalized
POVMs concentrated on finite subsets is dense in $\px$.
\end{proposition}
\begin{proof}
Let $\mu\in\px$, and $E$ be a typical open set in $\px$ containing $\mu$ of the form
\begin{align*}
    E=\left\{\nu\in\px; \left|\int_Xf_i d\nu_{h_i,k_i}-\int_X f_id\mu_{h_i,k_i}\right|<\epsilon, 1\leq i\leq n\right\},
\end{align*}
for some fixed $f_i\in\cbx$, $h_i,k_i\in\h$, $i=1,\ldots,n$ and
$\epsilon>0$. We shall obtain   an element in $E$ concentrated on a finite subset, which will imply the required result. Now for each $i\in \{1,\ldots,n\}$,  get simple
functions $g_i$ on $X$ satisfying
\begin{align*}
    \sup_{x\in X}|f_i(x)-g_i(x)|<\epsilon/2M,
\end{align*}
where $M$ is a positive constant with $M>\sup_i\|h_i\|\|k_i\|$. Since $g_i$'s are simple functions, there is a finite partition $\{A_{ij}\}$ of $X$ and scalars $\{c_{ij}\}\subseteq\C$ (where $j$ varies over some finite indexing set, say $\Lambda_i$  for each $ 1\leq i\leq n$)  such that
\begin{align*}
    g_i=\sum_{j\in \Lambda_i}c_{ij}\chi_{A_{ij}}
\end{align*}
for each $i$. Pick $x_{ij}\in A_{ij}$ and set
\begin{align*}
    \nu=\sum_{i=1}^n\sum_{j\in \Lambda_i}\delta_{x_{ij}}(\cdot)\mu(A_{ij}).
\end{align*}
It is clear that $\nu$ is a POVM concentrated on a finite subset. Also we have
$$\nu(X)=\sum_{i=1}^n\sum_{j\in \Lambda_i}\mu(A_{ij})=\mu(X)=I_\h,$$
 and hence $\nu$ is normalized. We claim that $\nu\in E$. Firstly note that
$$\int_X fd\nu=\sum_{i=1}^n\sum_{j\in \Lambda_i}f(x_{ij})\mu(A_{ij})$$
for  any bounded Borel measurable  function $f$ on $X$ (here
$\int_Xfd\nu\in \bh$ is the operator  satisfying  $\left\langle
h,\left(\int_Xfd\nu \right)k\right\rangle=\int_Xf\nu_{h,k}$ for all
$h,k\in \h$). Therefore for each $m\in \{1,\ldots,n\}$, we have
\begin{align*}
    \int_Xg_md\nu= \sum_{i=1}^n\sum_{j\in \Lambda_i}g_m(x_{ij})\mu(A_{ij})
   =\sum_{j\in \Lambda_m}c_{mj}\mu(A_{mj})=\int_Xg_md\mu.
\end{align*}
Thus we get the following:
\begin{align*}
    \left|\int_X f_i d\nu_{h_i,k_i}-\int_X f_i d\mu_{h_i,k_i}\right|&\leq  \left|\int_X f_i d\nu_{h_i,k_i}-\int_X g_i d\nu_{h_i,k_i}\right|
    +\left|\int_X g_i  d\nu_{h_i,k_i}-\int_X g_i d\mu_{h_i,k_i}\right|\\
    &\;\;\;\;\;\quad\quad+\left|\int_X g_i d\mu_{h_i,k_i}-\int_X f_i d\mu_{h_i,k_i}\right|\\
    &\leq\int_X|f_i-g_i|\; d|\nu_{h_i,k_i}|+\int_X|g_i-f_i|\; d|\mu_{h_i,k_i}|\\
    &\leq \left(\sup_{x\in X}|f_i(x)-g_i(x)|\right) \left(|\nu_{h_i,k_i}|(X)+|\mu_{h_i,k_i}|(X)\right)\\
    &\leq \left(\epsilon/2M\right)(2\|h_i\|\|k_i\|)\\
    &<\epsilon
\end{align*}
for $i=1,\ldots,n$, where $|\mu_{h_i,k_i}|$ and $|\nu_{h_i,k_i}|$ denote the total variation of the complex measures $\mu_{h_i,k_i}$ and $\nu_{h_i,k_i}$ respectively and we have used the fact that $|\mu_{h_i,k_i}|(X)\leq \|h_i\|\|k_i\|$, which is straightforward to verify.  It then follows that $\nu\in E,$ completing the
proof.
\end{proof}

\begin{definition}\label{definition of C*-convex hull}
For a given subset $\mathcal{M}$ of $\px$, the {\em $\cst$-convex hull} of $\mathcal{M}$ is the set defined by
\begin{align}
    \left\{\sum_{i=1}^nT_i^*\mu_i(\cdot)T_i: \; \mu_i\in \mathcal{M}, T_i\in \bh\;\text{ for } 1\leq i\leq n \text{ such that } \sum_{i=1}^nT_i^*T_i=I_\h\right\}.
\end{align}
\end{definition}

\begin{theorem}\label{Krein-Milman type theorem for POVM}(Krein-Milman type Theorem)
Let $X$ be a topological space and $\h$ a Hilbert space. Then the $C^*$-convex hull of
Dirac spectral measures (i.e. $\delta_x(\cdot)I_\h$ for $x\in X$) is dense  in $\px$. In
particular, the $C^*$-convex hull of all $C^*$-extreme points is
dense in $\px$.
\end{theorem}
\begin{proof}
Fix $\mu\in\px$. By Proposition \ref{prop:normalized POVM with
finite support are dense}, there is  a net $\mu_i\in\px$ such that
$\mu_i\to\mu$ in $\px$ and each $\mu_i$ is concentrated on a finite
subset. Therefore if we show that each $\mu_i$ is in the $C^*$-convex
hull of Dirac spectral measures, then we are done. So
assume without loss of generality, that $\mu\in\px$  is concentrated on a finite subset, say $\{x_1,\ldots,x_n\}$. If $T_i=\mu(\{x_i\})$, then
it is immediate that
$$\mu=\sum_{i=1}^n\delta_{x_i}(\cdot)T_i.$$
Set $S_i={T_i}^{1/2}\in\bh$ for each $i$. Then
\begin{align*}
    \sum_{i=1}^nS_i^*S_i=\sum_{i=1}^nT_i=\mu(X)=I_\h
\end{align*}
and
\begin{align*}\mu (\cdot ) =\sum_{i=1}^n S_i^*\delta_{x_i}(\cdot )S_i,
\end{align*}
which confirms that $\mu$ is a $C^*$-convex combination of Dirac
spectral measures.
\end{proof}

It is obvious that Dirac spectral measures are regular.
Therefore, Theorem \ref{Krein-Milman type theorem for POVM} along with Proposition \ref{a regular
povm is C*-extreme in px iff it is in rpx} give us the following
version of Krein-Milman theorem for regular POVMs. Its usefulness
shall be apparent when we discuss unital completely positive maps in
the next section.

\begin{corollary}\label{krien-milman theorem for regular POVM}
Let $X$ be a topological space and $\h$ a Hilbert space. Then the $\cst$-convex hull of
all regular spectral measures (in particular, regular $\cst$-extreme points)  is dense in $\rpx$.
\end{corollary}

\section{Applications to Completely Positive Maps}\label{application to completely positive maps}

We now apply the results  we have obtained in previous sections for
POVMs, to the theory of completely positive maps on  unital
commutative $\cst$-algebras. That there is a strong relationship
between these two topics is folklore.

If $\A$ is a commutative unital  $\cst$-algebra then by
Gelfand-Naimark theorem, there is a compact Hausdorff space $X$ (called {\em spectrum} of $\A$) such
that $\A=C(X)$, the space of all continuous functions on $X$.
Therefore for the rest of this section, we assume that $X$ is a compact
Hausdorff space.

\subsection{Completely positive maps}

Like before let $\bh $ be the algebra of all bounded operators on a
Hilbert space $\h.$ For any $C^*$-algebra ${\mathcal A}$, a linear
map $\phi:{ \mathcal A}\to \bh $ is called {\em positive}  if $\phi(a)\geq 0$ in $\bh$ whenever  $a\geq 0$ in $\A$. The map $\phi$ is called a \emph{completely
positive (CP) map} if $\phi \otimes id _n: {\mathcal A}\otimes M_n\to \bh
\otimes M_n$ is a positive map for every $n\in {\mathbb N}$ (here,
$id _n$ stands for the identity map on $n\times n$ matrix algebra $M_n$).
The well-known Stinespring's theorem (Theorem 4.1, \cite{Paulsen})
ensures that, if $\phi:{\mathcal A}\to\bh$ is a
completely positive map, then there exists a triple $(\psi,V,\K)$
where $\K$ is a Hilbert space, $\psi:\A\to\bk$ is a unital
$*$-homomorphism and $V\in\B(\h,\K)$  such that
\begin{equation}\label{Stinespring dilation theorem}
\phi (a) =V^*\psi (a)V~~\text{for all } a\in {\mathcal A},
\end{equation} and
satisfies the minimality condition: $\K=[\psi ({\mathcal A })V\h]$.
Moreover any such triple is unique up to unitary equivalence. In our
case, the algebra $C(X)$ being commutative, complete positivity of linear maps
on $C(X)$ is same as positivity (Theorem 3.11, \cite{Paulsen}).

\subsection{Correspondence between POVMs and CP maps}

Let $X$ be a compact Hausdorff space and $\h$ a Hilbert space. We now review the
correspondence between regular POVMs on $X$  and completely positive
maps on $\cx$ (see Chapter 4, \cite{Paulsen}). Because most of the
subsequent results hinge upon this correspondence, we give a
detailed description.

Given a regular POVM $\mu:\ox\to\bh$, consider
for any $f\in\cx$, the map $B_f:\h\times\h\to\C$ defined by
\begin{align*}
    B_f(h,k)=\int_{X}fd\mu_{h,k}~~\text{for all }h,k\in \h
\end{align*}
where $\mu_{h,k}$ denotes the  complex measure, as in
\eqref{eq:notation for mu_h,k}. It is straightforward to check that
$B_f$ is a sesquilinear form satisfying $\|B_f\|\leq
\|f\|\|\mu(X)\|$ and therefore, by Riesz Theorem (\emph{Theorem
II.2.2, \cite{Conway}}) we obtain a unique bounded operator, call
it $\phi_\mu(f)\in\bh$, satisfying $B_f(h,k)=\langle
h,\phi_\mu(f)k\rangle$. Further it is immediate that
$\phi_\mu(f)\geq0$ in $\bh$, whenever $f\geq0$ in $\cx$. Hence, the
induced map $\phi_\mu:C(X)\to\bh$  defines a completely positive map
via the assignment
\begin{align}
    \langle h,\phi_\mu(f)k\rangle=\int_X f d\mu_{h,k}~~~ \text{for all }f\in C(X)\text{ and }h,k\in\h.\label{eq:correspondece of phi and mu}
\end{align}
On the other hand, given a completely  positive map
$\phi:C(X)\to\bh$,  consider for each $h,k\in\h,$ the bounded linear
functional on $\cx:= f\mapsto \langle h,\phi(f)k\rangle$. Then by
the application of Riesz-Markov representation theorem, we obtain a
unique regular Borel measure  $\nu_{h,k}$ satisfying
\begin{align*}
    \langle h,\phi(f)k\rangle =\int_Xfd\nu_{h,k}~~\text{for all }f\in\cx
\end{align*}
and $\|\nu_{h,k}\|\leq \|\phi\|\|h\|\|k\|$. Now for each bounded Borel
measurable   function $g$, consider the map: $(h,k)\mapsto
\int_{X}gd\nu_{h,k}$ from $\h\times\h$ to $\C$, which is
sesquilinear as above and bounded by $\|\phi\|\|g\|.$ Hence again by
Riesz Theorem,  we obtain a unique bounded operator
$\tilde{\phi}(g)\in\bh$ satisfying
\begin{align}
    \langle h,\tilde{\phi}(g)k\rangle=\int_{X}gd\nu_{h,k}~~\text{for all }h,k\in\h.
\end{align}
Note that $\tilde{\phi}(g)\geq 0$ in $\bh$ whenever $g\geq0$ in B(X),
 the collection of all bounded Borel measurable   functions on $X$. In particular  for $A\in\ox$, if we set
\begin{align}
    \mu_\phi(A)=\tilde{\phi}(\chi_A),
\end{align}
where $\chi_A\in B(X)$ is the characteristic  function of the subset $A$,
then $\mu_\phi(A)$ is a positive operator in $\bh$  and satisfies
\begin{align*}
    \nu_{h,k}(A)=\langle h,\mu_\phi(A)k\rangle~~\mbox{for all }h,k\in\h.
\end{align*}
Because  $\nu_{h,h}$ is a regular Borel  positive measure for each $h\in \h$,
it is  immediate  that $\mu_\phi$ defines a regular POVM   which satisfies
the equality $\mu_\phi(X)={\phi}(1)$,
where $1$ denotes the
constant function 1 on $X$.

\begin{remark}
For any POVM (not necessarily regular)
$\mu$, one can define a completely  positive map $\phi$ satisfying
 $\eqref{eq:correspondece of phi and mu}$ in a similar way. However, the
regular measure $\mu_\phi$ corresponding to $\phi$ that we got
above, could significantly be different than the original $\mu$.
More precisely, there may exist more than one Borel
POVM on a compact Hausdorff space $X$ (certainly, non-metrizable), say $\mu_1$ and $\mu_2$, such that
$\int_Xfd\mu_1=\int_Xfd\mu_2$ for all $f\in \cx$ (see the discussion just before Remark \ref{X is compact iff rpx is compact}). Therefore to
maintain  uniqueness, we shall always assume the POVM to be regular
whenever we talk about the correspondence between a POVM and a
 completely positive map.
\end{remark}

The following theorem summarises some basic properties of this
correspondence. See (Proposition 4.5, \cite{Paulsen}), \cite{Hadwin}, \cite{Han Larson Liu} for some
discussions on this.

\begin{theorem}\label{thm:correspondece between POVM and cp maps}
Let $X$ be a compact Hausdorff space and let $\h$ be a  Hilbert
space. Then  the correspondence described above between $\bh $ valued regular POVMs on $X$ and completely positive maps on $\cx$,
satisfies the following:
\begin{enumerate}
 \item \label{eq: uniqueness of mu and phi} $\phi_{\mu_\phi}=\phi$ and $\mu_{\phi_\mu}=\mu$.
 \item\label{eq:expression of mu(X) and phi(1)} $\phi(1)=\mu_\phi(X)$.
 \item\label{eq:homomorphism and PVM equivalence}  $\mu$ is a projection valued measure if and only if $\phi_\mu$ is a $*$-homomorphism.
 \item\label{eq:sum of two phi and mu} $\phi_{\mu_1+\mu_2}=\phi_{\mu_1}+\phi_{\mu_2}$ and $\mu_{\phi_1+\phi_2}=\mu_{\phi_1}+\mu_{\phi_2}$.
 \item\label{eq:operator sum and multiplication} $\phi_{T^*\mu(\cdot)T}=T^*\phi_\mu(\cdot)T$ and $\mu_{T^*\phi(\cdot)T}=T^*\mu_\phi(\cdot)T$ for any $T\in\bh$.
\end{enumerate}
\end{theorem}
\begin{proof}
Part \eqref{eq: uniqueness of mu and phi} is just uniqueness of the
correspondence and part \eqref{eq:expression of mu(X) and phi(1)}
follows from the discussion above.
To show \eqref{eq:homomorphism and PVM equivalence}, first assume
that $\phi_\mu$ is a $*$-homomorphism. Then for all $f,g\in\cx$ and
$h,k\in\h$, we have
\begin{align*}
    \int_X fgd\mu_{h,k}=\langle h,\phi_\mu(fg)k\rangle=\langle h,\phi_\mu(f)\phi_\mu(g)k\rangle=\int_X fd\mu_{h,\phi_\mu(g)k}.
\end{align*}
Since $f\in\cx$ is arbitrary, it follows from uniqueness of regular Borel  measures in Riesz-Markov theorem that $gd\mu_{h,k}=d\mu_{h,\phi_\mu(g)k}$, as complex measures. Equivalently for any $A\in\ox$, we have
\begin{align*}
    \int_A gd\mu_{h,k}=\mu_{h,\phi_\mu(g)k}(A),
\end{align*}
that is
\begin{align*}
    \int_X g\chi_Ad\mu_{h,k}= \langle h,\mu(A)\phi_\mu(g)k\rangle=\langle\mu(A)h, \phi_\mu(g)k\rangle =\int_X gd\mu_{\mu(A)h,k}.
\end{align*}
Again, since $g\in\cx$ is arbitrary, we conclude  that $\chi_Ad\mu_{h,k}=d\mu_{\mu(A)h,k}$, as complex measures. Equivalently for any $B\in\ox$, we get
\begin{align*}
   \int_{X}\chi_{A\cap B}d\mu_{h,k}= \int_X\chi_A\chi_Bd\mu_{h,k}=\mu_{\mu(A)h,k}(B)=\langle \mu(A)h,\mu(B)k\rangle,
\end{align*}
which further implies
$$\langle h,\mu(A\cap B)k\rangle=\langle h,\mu(A)\mu(B)k\rangle.$$
Since  $h,k\in\h$ are arbitrary, we conclude that
$$\mu(A\cap B)=\mu(A)\mu(B)~~\text{for all }A,B\in\ox,$$
which shows that $\mu$ is a projection valued measure. The converse of the statement follows just by  reversing of the argument above.

Part \eqref{eq:sum of two phi and mu} directly follows  from the assignment in  \eqref{eq:correspondece of phi and mu}. To show part \eqref{eq:operator sum and multiplication}: let $T\in\bh$ and set $\nu(\cdot)=T^*\mu(\cdot)T$. For any $h,k\in\h$ and $B\in\ox$, then
\begin{align*}
   \langle h,\nu(B)k\rangle= \langle h,T^*\mu(B)Tk\rangle=\langle Th,\mu(B) Tk\rangle
\end{align*}
which equivalently says  $\nu_{h,k}=\mu_{Th,Tk}$, as complex measures. Therefore for any $f\in\cx$, we have
\begin{align*}
    \langle h,\phi_{\nu}(f)k\rangle=\int_Xfd\nu_{h,k}=\int_Xfd\mu_{Th,Tk}=\langle Th,\phi_\mu(f)Tk\rangle=\langle h,T^*\phi_\mu(f)Tk\rangle
\end{align*}
which proves that $\phi_{\nu}=T^*\phi_\mu(\cdot) T$. The other equality follows similarly.
\end{proof}

It is crucial  that for a compact Hausdorff space $X$, if $\mu$ is
a regular POVM with a Naimark  dilation $(\pi,V,\hpi)$ then
$(\phi_\pi,V,\hpi)$ is a Stinespring dilation for the corresponding
CP map $\phi_\mu$ (follows directly from part (5) of Theorem
\ref{thm:correspondece between POVM and cp maps}). Further,
minimality conditions match:
\begin{align}
    [\pi(\ox)V\h]=[\phi_\pi(C(X))V\h]
\end{align}
and therefore, the Stinespring dilation $\phi_\mu=V^*\phi_\pi(\cdot) V$ is
minimal  if and only if the Naimark dilation $\mu=V^*\pi(\cdot) V$ is
minimal. Here we have some additional technical properties of this
correspondence which are quite useful for us.

\begin{proposition}\label{prop:commutant of pi and phi_pi are same}
Let $X$ be a compact Hausdorff space and $\mu:\ox\to\bh$ a regular
POVM. Then $\mu(\ox)'=\phi_\mu(C(X))'$.
Moreover, $\mu(A)\in$ WOT-$\overline{\phi_\mu(C(X))}$ and
$\phi_\mu(f)\in$ WOT-$\cspan\mu(\ox)$ for all $A\in\ox$ and $f\in
C(X)$ and in particular, WOT-$\overline{\phi_\mu(\cx)}$=WOT-$\overline{\Span}{\mu(\ox)}$.
\end{proposition}
\begin{proof}
First assume $T\in\mu(\ox)'$.  Then $\mu(A)T=T\mu(A)$ for all $A\in\ox$ and hence
\begin{align*}
    \langle T^*h,\mu(A)k\rangle=\langle h, T\mu(A)k\rangle=\langle h,\mu(A)Tk\rangle,
\end{align*}
for all $h,k\in\h$, which  is equivalent to
$\mu_{T^*h,k}=\mu_{h,Tk}$, as complex measures. Therefore for all $f\in C(X)$, it follows that
\begin{align*}
\langle T^*h,\phi_\mu(f)k\rangle =\int_X fd\mu_{T^*h,k}=\int_X fd\mu_{h,Tk}=\langle h,\phi_\mu(f)Tk\rangle.
\end{align*}
 Since $h,k\in\h$ are arbitrary, we conclude that
 $$T\phi_\mu(f)=\phi_\mu(f)T ~~\text{for all }f\in\cx,$$
 which implies $T\in\phi_\mu(\cx)'$. Thus we have proved the inclusion $\mu(\ox)'\subseteq\phi_\mu(\cx)'$. The other way of the inclusion is similarly proved just by reversing the  implications above.

Now let $(\pi,V,\hpi)$ be the minimal Naimark dilation for $\mu$. To show that $\mu(A)\in$ WOT-$\overline{\phi_\mu(\cx)}$ for $A\in\ox$,  firstly note that
$$\pi(\ox)''=\phi_\pi(\cx)'',$$
the double commutant of the respective sets in $\B(\hpi)$, which follows from first part of the proof. Therefore, since $\pi(A)\in\pi(\ox)$ and $\pi(\ox)\subseteq\pi(\ox)''=\phi_\pi(C(X))''$, it follows from double commutant theorem (Theorem IX.6.4, \cite{Conway})  for the $*$-algebra $\phi_\pi(C(X))$, that there is  a net $\{f_i\}$ in $C(X)$ such that
$$\phi_\pi(f_i)\to\pi(A) ~~\text{ in WOT}.$$
This implies
\begin{align*}
    \phi_\mu(f_i)=V^*\phi_\pi(f_i)V\to V^*\pi(A)V=\mu(A) ~~\text{ in WOT}
\end{align*}
 and so we conclude that $\mu(A)\in$ WOT-$\overline{\phi_\mu(C(X))}$. Other assertions follow similarly.
\end{proof}

\subsection{$\cst$-extreme points of UCP maps on commutative $\cst$-algebras}
For a unital $C^*$-algebra $\A$ and a Hilbert space $\h$,  let
$UCP_\h(\A)$ denote the collection of all unital  completely
positive maps from $\A$ to $\bh$. It is clear that $UCP_\h(\A)$ is a convex set. The
seminal paper by Arveson \cite{Arveson1} studies the extreme points
of $UCP_\h(\A)$ and provides an abstract characterization. Several
authors have looked into the classical convexity (\cite{Choi},
\cite{KRP4}, \cite{KRP1}, \cite{Stormer} and \cite{Bhat Pati
Sundar} ) of $UCP_\h(\A)$ and its subclasses. Many others have
considered different versions of  convexity on $UCP_\h(\A)$, e.g.  \cite{Effros
Winkler}, \cite{Fujimoto},  \cite{Davidson Kennedy}, \cite{Loebl
Paulsen}, \cite{Magajna}, \cite{Farenick Morenz}, \cite{Farenick
Zhou}, \cite{Zhou} and \cite{Gregg}.

The main focus of this paper has been on  the notion of
$\cst$-convexity. Farenick and Morenz \cite{Farenick Morenz} first
studied the $\cst$-convexity structure of $UCP_\h(\A)$. They gave a complete
characterization of all $\cst$-extreme points of $UCP_\h(\A)$,
whenever $\h$ is finite dimensional. In \cite{Farenick Zhou}, an
abstract characterization of all $\cst$-extreme points were given,
which we have presented in Theorem \ref{thm:Farenick and Zhou
characterization of $C^*$-extreme points} in the language of POVMs.
In \cite{Gregg}, Gregg obtained a necessary criterion for
$\cst$-extreme points in $UCP_\h(\A)$, when $\A$ is a commutative
unital $\cst$-algebra. We carry forward this investigation of
$\cst$-convexity structure of $\uhx$ by using the tools that we have
developed for POVMs and its correspondence with completely positive
maps.

More formally,  $UCP_\h(C(X))$ is a $C^*$-convex set in the sense that
\begin{align*}
\sum_{i=1}^nT_i^*\phi_i(\cdot)T_i\in UCP_\h(\cx)
\end{align*}
whenever  $\phi_i\in UCP_\h(\cx)$ and $T_i\in\bh$ with
$\sum_{i=1}^nT_i^*T_i=I_\h$. In a way similar to $\cst$-extreme
points for POVMs in Definition \ref{definition of C*-extreme
points}, we  define  $C^*$-extreme points of $UCP_\h(\cx)$ (see
\cite{Farenick Morenz}) as follows:

\begin{definition}\label{C*-convexity for UCP maps}
A map $\phi\in\uhx$ is {\em $\cst$-extreme} if, whenever
$\phi=\sum_{i=1}^nT_i^*\phi_i(\cdot)T_i$ for $\phi_i\in\uhx$ with
invertible operators $T_i\in\bh$ satisfying
$\sum_{i=1}^nT_i^*T_i=I_\h$, then  $\phi _i$ is unitarily equivalent to
$\phi$ i.e. $\phi=U_i^*\phi_i(\cdot)U_i$ for some unitary operator
$U_i\in\bh$ for every $i.$
\end{definition}

The correspondence of regular
 POVMs and completely positive maps  described above clearly
 preserves classical as well as   $\cst$-convexity structures. Recall that $\rpx$ denotes the collection
 of all regular Borel  normalized POVMs on $X$.

\begin{theorem}\label{presevance of C*-extreme point in the correspondece of POVM and cp maps}
A normalized regular POVM $\mu$ is
$C^*$-extreme (extreme) in $\rpx$ (or in $\px$) if and only if
$\phi_\mu$ is $C^*$-extreme (extreme) in $UCP_\h(\cx)$.
\end{theorem}
\begin{proof}
The proof follows from Theorem \ref{thm:correspondece between POVM
and cp maps}, because   classical, $C^*$-convex combinations and
unitary equivalences are preserved under the correspondence.
\end{proof}

Following the discussions above, we are now ready to deduce some
results  for $\uhx$. As noticed in Proposition \ref{a regular povm
is C*-extreme in px iff it is in rpx}, a regular normalized POVM
$\mu$ is a $\cst$-extreme  point in $\px$ if and only if $\mu$ is a
$\cst$-extreme point in $\rpx$. Therefore, it follows from Theorem
\ref{presevance of C*-extreme point in the correspondece of POVM and cp maps} that $\mu$ is  $\cst$-extreme in $\px$ if and only if
$\phi_\mu$ is $\cst$-extreme in $\uhx$. Thus, whenever $X$ is a
compact Hausdorff space, we have got freedom to bring back all the
results on $\cst$-extreme point in $\px$ into  the language of
$\uhx$. We frequently make use of Theorem \ref{thm:correspondece
between POVM and cp maps} and Theorem \ref{presevance of C*-extreme
point in the correspondece of POVM and cp maps}. Before going
forward, we recall the following known fact.

\begin{theorem}(Proposition 1.2, \cite{Farenick Morenz})
Every unital $*$-homomorphism is a $\cst$-extreme point in $UCP_\h(C(X))$.
\end{theorem}

 Now let $X$ be a countable compact Hausdorff space. Then we saw in Theorem
\ref{atomic $C^*$-extreme points are PVM} that every $\cst$-extreme
point in $\px$ is  spectral. Since spectral measures correspond to unital $*$-homomorphisms, here is the corresponding result.
\begin{theorem}\label{$C^*$-extreme UCP maps are homomorphisms}
Let $\A$ be a commutative unital $C^*$-algebra with countable
spectrum and let $\phi $ be a map in $UCP_\h(\A)$. Then $\phi$ is
$C^*$-extreme if and only if $\phi$ is a $*$-homomorphism.
\end{theorem}
We apply this result to the $C^*$-algebra generated by a single
normal operator to have the following.

\begin{example}\label{ucp map induced from a normal operator}
Let $N\in\bk$ be a normal operator on a Hilbert space $\K$ with
countable spectrum $\sigma(N)$ (in particular, when $N$ is compact).
It is known that for such a normal operator, a subspace
$\h\subseteq\K$ is invariant for $N$ if and only if it is reducing
for $N$(Theorem 1.23, \cite{Rajdavi Rosenthal}). Consider the unital
completely positive map $\phi_N:C^*(N)\to\bh$ defined by
$\phi_N(T)=P_\h T_{|_\h}$ for all $T\in C^*(N)$, where $C^*(N)$ is the
unital $C^*$-algebra generated by $N.$ It is easy to verify that $\phi_N$ is a $*$-homomorphism if and only if $\h$ is a reducing subspace for $N$. Thus since $C^*(N)$ is isomorphic
to $C(\sigma(N))$ as $\cst$-algebra and $\sigma(N)$ is countable,
the argument above along with Theorem \ref{$C^*$-extreme UCP maps
are homomorphisms} show that the following conditions are
equivalent:
\begin{enumerate}
    \item  $\phi_N$ is a $C^*$-extreme point in $UCP_\h(C^*(N))$.
\item  $\phi_N$ is a $*$-homomorphism.
\item  $\h$ is an invariant subspace of $N$.
\item $\h$ is a co-invariant subspace of $N$.
\item  $\h$ is a reducing subspace of $N$.
\end{enumerate}
\end{example}

Next using the results in Section \ref{measure isomorphism}, we provide here an
 example of a $\cst$-extreme point in $\px$ which is not spectral, whenever $X$ is an uncountable compact metric space and $\h$ an infinite dimensional Hilbert space.

\begin{example}\label{existence of a non-homomorphic C^*-extreme points on an uncountable metric space}
 Consider the normalized POVM  $\nu:\mathcal{O}(\T)\to\mathcal{B}(\h^2)$  defined by
\begin{align*}
\nu(A)=P_{\h^2}{M_{\chi_A}}_{|_{\h^2}}~~\text{for all }A\in\mathcal{O}(\T),
\end{align*}
where $\h^2$ denotes the Hardy space on the unit circle $\mathbb{T}$.
Here $M_{f}$ denotes the multiplication operator on $L^2(\T)$ for any $f\in L^\infty(\T)$.
Then the corresponding unital completely positive map $\phi_\nu:C(\T)\to\B(\h^2)$ is given by
$$\phi_\nu(f)=P_{\h^2}{M_{f}}_{|_{\h^2}}~~\text{for all }f\in C(\T).$$
It is known  (Example 2, \cite{Farenick Morenz}) that $\phi_\nu$ is
a $C^*$-extreme point in $UCP_\h(C(\T))$ and therefore, $\nu$ is
$C^*$-extreme in $\mathcal{P}_{\h^2}(\T)$ by Theorem
\ref{presevance of C*-extreme point in the correspondece of POVM and
cp maps}. Also note that $\nu$ is  not  spectral, since $\phi_\nu$
is not a $*$-homomorphism.  Now let $X$ be an uncountable compact
metric space. Then by well-known theorems of Borel isomorphism
(Theorem 2.12, \cite{KRP3}),  there exists a Borel isomorphism
$f:\T\to X$.  Define the normalized POVM $\mu:\ox\to\B(\h^2)$ by
\begin{align}
    \mu(A)=\nu(f^{-1}(A))~~\text{for all }A\in\ox.
\end{align}
Then Theorem \ref{Borel isomorphism preserve same properties} along with Theorem \ref{some similar properties of isomorphic POVM} imply that $\mu$ is a $C^*$-extreme point in $\s_{\h^2}(X)$ and is not  spectral. Thus, since any infinite dimensional separable Hilbert  space is isomorphic to $\h^2$, what we have shown is that whenever $X$ is an uncountable compact metric space and $\h$ an infinite dimensional Hilbert space, then $\px$ contains a $\cst$-extreme point which is not spectral. The  assertion above can be applied to Polish spaces as well.
\end{example}

 Let $E$ be an uncountable compact subset of $\C$. Then $E$ is a
compact metric space.  We consider the normalized POVM $\mu:\mathcal{O}(E)\to\B(\h^2)$ constructed in Example \ref{existence of a non-homomorphic C^*-extreme points on an uncountable metric space},
 which is already in the minimal Naimark
dilation form $\mu (\cdot ) =V^*\pi (\cdot ) V$. If $N=\int_E
zd\pi\in\B(\hpi)$, then $N$ is a normal operator with spectrum
$E$. Also the corresponding completely positive map
$\phi_\mu:C^*(N)\to\B(\h^2)$ is of the form $\phi_\mu(T)=P_{\h^2} T_{|_{\h^2}}$ for $T\in C^*(N)$.
Thus we have got an example of a completely positive map of the
form $\phi_N$ as discussed in Example \ref{ucp map induced from a normal operator}, which is $C^*$-extreme but not a
$*$-homomorphism.

Now let $\A$ be a separable commutative unital $C^*$-algebra. Then its spectrum  is a separable compact Hausdorff space (Theorem V.6.6, \cite{Conway}) and hence metrizable, which is to say $\A=\cx$ for a compact metric space $X$. Therefore, Example \ref{existence of a non-homomorphic C^*-extreme points on an uncountable metric space} and Theorem \ref{presevance of C*-extreme point in the correspondece of POVM and cp maps}  give us the following result for a separable commutative unital $\cst$-algebra with uncountable spectrum.

\begin{theorem}\label{a non homomorphic C*-extreme point for separable C*-algebra}
Let $\A$ be a separable commutative unital $C^*$-algebra with
uncountable spectrum and let $\h$ be an
 infinite dimensional separable Hilbert space. Then $UCP_\h(\A)$ contains a $C^*$-extreme point which is not a $*$-homomorphism.
\end{theorem}

The theorem above fails to be true if the separability assumption is removed, as we see below. If $X$ is a discrete space and $\tilde{X}$ denotes its one-point compactification, then we saw in
Remark \ref{C*-extreme point on one-point compactification of discrete space} that  every regular POVM in $\mathcal{P}_\h(\tilde{X})$ is atomic, and hence every $\cst$-extreme point in $\mathcal{R}\p_\h(\tilde{X})$ is  spectral. Equivalently, every $\cst$-extreme point in $UCP_\h(C(\tilde{X}))$ is a $*$-homomorphism by Theorem \ref{presevance of C*-extreme point in the correspondece of POVM and cp maps}. Note that, whenever $X$ is an uncountable discrete space,  then $\tilde{X}$ is a non-separable compact Hausdorff space  and in particular, $C(\tilde{X})$ is a non separable $\cst$-algebra (Theorem V.6.6, \cite{Conway}). Thus the assumption of separability of the $\cst$-algebra $\A$  in Theorem \ref{a non homomorphic C*-extreme point for separable C*-algebra} is crucial. We have obtained the following:

\begin{theorem}
Let $\A$ be a commutative unital  $\cst$-algebra whose spectrum is
 a one-point compactification of a discrete space.
Then every $\cst$-extreme point in $UCP_\h(\A)$ is a $*$-homomorphism.
\end{theorem}

Next let $\phi:\cx\to\bh$ be a unital completely positive map such
that $\phi(\cx)$ is commutative.  Then
WOT-$\overline{\phi(\cx)}$ is commutative. Since
WOT-$\overline{\phi(\cx)}=$ WOT-$\overline{\Span}\mu_\phi(\ox)$ by
Proposition \ref{prop:commutant of pi and phi_pi are same}, it
follows that WOT-$\overline{\Span}\mu_\phi(\ox)$ is commutative. In
particular, $\mu_\phi(\ox)$ is commutative. Therefore if $\phi$ is a
$\cst$-extreme point in $\uhx$ with commutative range, then
$\mu_\phi$ is a $\cst$-extreme point in $\px$ with commutative range.
Then it follows from Theorem \ref{commutative C*-extreme points are
PVM} that $\mu_\phi$ is spectral and hence, $\phi$ is a
$*$-homomorphism. Thus we have got the following result. A similar
result for extreme points with commutative range in $UCP_\h(\cx)$
holds true (see Corollary 3.6, \cite{Stormer}).

\begin{theorem}\label{C-extreme UCP maps with commutative ranges are homomorphism}
Let $\A$ be a commutative unital  $\cst$-algebra and $\phi:\A\to\bh$ a unital completely positive map with commutative range. Then $\phi$ is $\cst$-extreme in $UCP_\h(\A)$ if and only if $\phi$ is a $*$-homomorphism.
\end{theorem}

We now discuss the bounded-weak topology on $\uhx$ and how it is connected to the topology on $\px$ defined earlier (which we called bounded weak topology as well). The bounded-weak topology (see \cite{Arveson1}, \cite{Paulsen}) on $UCP_\h(C(X))$
 is given by the convergence: for a net $\{\phi_i\}$ and $\phi$ in $UCP_\h(\cx)$,
\begin{align*}
    \phi_i\to\phi\text{  if and only if } \phi_i(f)\to\phi(f) \text{ in WOT}
\end{align*}
 for all  $f\in C(X).$

 For a net $\mu^i$ and $\mu\in \rpx$, since $\phi_\mu(f)=\int_Xfd\mu$ for all $f\in\cx$, it follows that  $\mu^i\to\mu$ in $\rpx$ if and only if $\phi_{\mu^i}(f)\to\phi_\mu(f)$ in WOT for all $f\in \cx$.
 The following proposition is just a rephrasing  of the definition of the topology on regular POVMs, which effectively says that $\rpx$ and $\uhx$ are topologically homeomorphic. Recall that by Riesz-Markov representation theorem, the space of all regular Borel  complex measures $M(X)$ on $X$ is Banach space dual of $C(X)$.

\begin{proposition}
Let $\mu^i$ be a net in $\rpx$ and $\mu\in\rpx$. Then the following are equivalent:
\begin{enumerate}
  \item $\mu^i\to\mu$ in $\rpx$ (and, in $\px$).
   \item $\phi_{\mu^i}\to\phi_\mu$ in bounded-weak topology in $UCP_\h(C(X))$.
    \item $\mu^i_{h,k}\to\mu_{h,k}$ in weak*-topology on $M(X)$ for all $h,k\in \h$.
\end{enumerate}
\end{proposition}

In \cite{Farenick Morenz}, a Krein-Milman type theorem was proved
for $UCP_\h(\A)$ with respect to bounded-weak topology, for
arbitrary unital $\cst$-algebra $\A$ but finite-dimensional Hilbert
space $\h$. Here we consider commutative unital $\cst$-algebras and
arbitrary Hilbert spaces and give a similar kind of result for
$\uhx$. As in Definition \ref{definition of C*-convex hull}, we
define the {\em $\cst$-convex hull} of a subset
$\mathcal{N}\subseteq\uhx$ by
\begin{align*}
    \left\{\sum_{i=1}^nT_i^*\phi_i(\cdot)T_i;~ \phi_i\in \mathcal{N}, T_i\in\bh \text{ such that }\sum_{i=1}^nT_i^*T_i=I_\h\right\}.
\end{align*}
 The following version of Krein-Milman type theorem for commutative unital $\cst$-algebra follows from Corollary  \ref{krien-milman theorem for regular POVM}, Theorem \ref{presevance of C*-extreme point in the correspondece of POVM and cp maps} and Theorem \ref{thm:correspondece between POVM and cp maps}.

\begin{theorem}\label{Krein-Milman theorem for UCP maps}
Let $\A$ be a commutative unital $C^*$-algebra and $\h$ a Hilbert space. Then the $C^*$-convex hull of the collection of all unital $*$-homomorphisms (in particular, $\cst$-extreme points)  is dense in $UCP_\h(\A)$ with respect to bounded-weak topology.
\end{theorem}

\section{Conclusion}

Our original interest was to study $\cst$-convexity and $\cst$-extreme points in the setting of unital completely positive maps on unital commutative $\cst$-algebras. For this purpose, we have taken recourse in the well-known correspondence between such  maps and POVMs on compact spaces. While doing so, we thought it could be of  independent interest to study  $\cst$-convexity in the setting of POVMs.  So we analyze  POVMs on general measurable spaces and get several  interesting and basic results. Naimark's dilation theorem plays a crucial role in our investigation of POVMs, just as Stinespring's dilation theorem does for completely positive maps. Below we highlight some of our main results.

The abstract characterizations of $\cst$-extreme POVMs in Theorem \ref{thm:Farenick and Zhou characterization of $C^*$-extreme points} and Corollary \ref{Zhou Characterization of $C^*$-extreme points} are the building blocks for all the forthcoming results. Our first major result is Theorem \ref{if mu(A) commutes with everything, then it is a projection}, which says that for a $\cst$-extreme POVM $\mu:\ox\to\bh$ and $E\in\ox$, if $\mu(E)$ commutes with $\mu(A)$ for all $A\subseteq E$, then $\mu(E)$ is a projection. The significance of this theorem should be clear from the  following consequences:
\begin{itemize}
\item All $\cst$-extreme POVMs with commutative ranges are spectral (Theorem \ref{commutative C*-extreme points are PVM}).
\item All atomic $\cst$-extreme POVMs are spectral, and hence all $\cst$-extreme POVMs on countable spaces are spectral (Theorem \ref{atomic $C^*$-extreme points are PVM}).
\item If $\dim\h<\infty$, then  all $\cst$-extreme POVMs
are spectral (Theorem \ref{POVMs on finite dimensional spaces are spectral}).
\end{itemize}
We next study mutually disjoint POVMs and behaviour of $\cst$-convexity under their direct sums. Here we show   the following:
\begin{itemize}
    \item Any $\cst$-extreme POVM decomposes uniquely
into a direct sum of an atomic $\cst$-extreme POVM and a non-atomic $\cst$-extreme POVM such that they are mutually disjoint (Theorem \ref{thm:every povm decomposes as direct sum of atomic and non atomic povm}).
\end{itemize}
In essence, this implies that in order to get complete picture of $\cst$-extreme POVMs, it suffices to understand non-atomic $\cst$-extreme POVMs, given the fact that we have already characterized atomic $\cst$-extreme POVMs.

Our next main result is a version of Krein-Milman theorem for the $\cst$-convexity of POVMs on topological spaces. We define an appropriate topology on the $\cst$-convex space $\px$ of normalized POVMs, and prove in Theorem \ref{Krein-Milman type theorem for POVM} that
\begin{itemize}
    \item $\px$ is closure of $\cst$-convex hull of the set of its $\cst$-extreme points.
\end{itemize}
Finally, we  apply our observations  about  POVMs on  compact Hausdorff spaces $X$ to the study of $\cst$-convexity of  the space $UCP_\h(C(X))$ of unital completely positive maps on the commutative  $\cst$-algebra $C(X)$. In particular, we have the following:
\begin{itemize}
\item If $X$ is countable (in particular, when $C(X)=\C^n$), then every $\cst$-extreme points of $UCP_\h(C(X))$ is a $*$-homomorphism (Theorem \ref{$C^*$-extreme UCP maps are homomorphisms}).
\item If $X$ is uncountable, then $UCP_\h(C(X))$ contains a $\cst$-extreme point which is not a $*$-homomorphism (Theorem \ref{a non homomorphic C*-extreme point for separable C*-algebra}).
\item All $\cst$-extreme points in $UCP_\h(C(X))$ with commutative ranges are  $*$-homomorphisms (Theorem \ref{C-extreme UCP maps with commutative ranges are homomorphism}).
\item (A Krein-Milman type theorem) The space $UCP_\h(C(X))$ is closure in bounded-weak topology of $\cst $-convex hull of its $\cst$-extreme points (Theorem \ref{Krein-Milman theorem for UCP maps}).
\end{itemize}


We mention here in the passing that the study of  POVMs on compact Hausdorff
spaces as done  in Section \ref{application to completely positive maps}   extends easily to
POVMs on locally compact Hausdorff spaces. Indeed if $X$ is a locally compact
non-compact Hausdorff space, then the set of  contractive POVMs:
$$    \cpx=\{\mu:\ox\to\bh; \mu \text{ is  a POVM and }\mu(X)\leq
    I_\h\}$$
forms a $C^*$-convex set. Any $\mu $ here extends to a normalized
POVM $\tilde{\mu}$ on the Borel $\sigma $-algebra of the one point compactification
 $\tilde {X}= X\bigcup \{ \infty \}$, by taking $\tilde {\mu }({\infty })=
 1-\mu (X).$ This correspondence between contractive POVMs on $X$
 and normalized POVMs on $\tilde {X}$ is bijective and preserves basic
 properties such as $C^*$-convexity, regularity, atomicity etc. Hence results
 can be easily translated back from the compact case.

 We conclude with a question. Our hope is that getting an  answer to this question may
  shed more  light on the structure of $C^*$-extreme points
 of UCP maps on commutative $\cst$-algebras with non-metrizable spectrum.
 We have shown that any $\cst$-extreme point in $\p_\h(\N)$ is
spectral, where $\N$ is the set of natural numbers. It is also known
that any unital completely positive map on $l^\infty(=l^\infty(\N))$
corresponds to finitely additive positive operator valued measure on
$\N$, whereas (countably additive) POVMs correspond to the normal CP
maps on $l^\infty$ and hence all normal $\cst$-extreme points are
$*$-homomorphic. It is not clear as of now how $\cst$-extreme points
in the collection of all finitely additive POVMs behave. Approaching
another way, the spectrum of $l^{\infty}$ is of course the
Stone-\v{C}ech compactification of ${\mathbb N}$. Unfortunately this
space is not metrizable and our result on existence of a
non-homomorphic $C^*$-extreme point (Theorem \ref{a non homomorphic
C*-extreme point for separable C*-algebra}) is not applicable and so
we are left with the following question:

\begin{question}
Are $C^*$-extreme unital completely positive maps on the
$C^*$-algebra $l^{\infty }$ always $*$-homomorphisms?
\end{question}

\noindent \textbf{Acknowledgements:}
We sincerely thank the referee for several constructive suggestions which helped improve the paper.
The first author was supported
by the national post-doctoral fellowship (NPDF)  of SERB (India)
with reference number PDF/2017/002554 and the Indian Statistical
Institute. The second author thanks J C Bose Fellowship, SERB(India)
for financial support.

\end{document}